\documentclass[runningheads]{llncs}
\usepackage[T1]{fontenc}
\usepackage{graphicx}
\usepackage[pdffitwindow=false,
            plainpages=false,
            pdfpagelabels=true,
            pdfpagemode=UseOutlines,
            pdfpagelayout=SinglePage,
            bookmarks=false,
            colorlinks=true,
            hyperfootnotes=false,
            linkcolor=blue,
            urlcolor=blue!30!black,
            citecolor=green!50!black]{hyperref}
\usepackage{color}

\usepackage{amsfonts}
\usepackage{amsmath}
\usepackage{dsfont}  % for \mathds{1} (double-struck 1)
\usepackage{enumitem}  % put this in the preamble
\usepackage{booktabs}  % for professional table formatting
\usepackage{wrapfig}  % for wrapping text around figures/algorithms
\usepackage{float}  % for [H] placement option to force figure position
\usepackage{subfig}  % for subfloat command to improve figure placement
\usepackage{algorithm}
\usepackage{algpseudocode}

\DeclareMathOperator*{\argmin}{arg\,min}
\usepackage{todonotes}

\DeclareMathOperator{\conv}{conv}

\DeclareMathOperator{\dist}{dist}

\newcommand{\innp}[2]{\left\langle #1, #2 \right\rangle}
\newcommand{\norm}[1]{\left\| #1 \right\|}
\newcommand{\vx}{\mathbf{x}}
\newcommand{\vvv}{\mathbf{v}}

\newcommand{\vy}{\mathbf{y}}

\begin{document}
\title{Graph Isomorphism: Mixed-Integer Convex Optimization from First-Order Methods}
\titlerunning{Graph isomorphism with mixed-integer convex optimization}
\author{Wenjie Xiao \inst{1,2} \and
Mathieu Besançon\inst{4,1} \and
Patrick Gel\ss\inst{2} \and
Deborah Hendrych\inst{1,2} \and
Stefan Klus\inst{3} \and
Sebastian Pokutta\inst{1,2}}
%s
\authorrunning{W. Xiao et al.}
\institute{Zuse Institute Berlin, Berlin, Germany \and
Technische Universität Berlin, Berlin, Germany \and
Heriot--Watt University, Edinburgh, UK \and
Univ.~Grenoble Alpes, Inria, CNRS, LIG, Grenoble, France}
\maketitle              % typeset the header of the contribution
\begin{abstract}
The \textit{graph isomorphism (GI) problem}, which asks whether two graphs are structurally identical,
occupies a unique position in computational complexity---it is neither known to be solvable in polynomial time, nor proven to be NP-complete.
We propose a convex mixed-integer formulation of the problem and leverage first-order convex optimization to tackle it,
following a stream of recent work on optimization-driven graph isomorphism detection. We strengthen our formulation with variable fixing techniques that prove highly effective while preserving the polyhedral structure.
We perform extensive computations evaluating the performance of different families of methods including a mixed-integer 
convex formulation, mixed-integer linear optimization, local search and spectral heuristics over a collection of challenging GI instances.
We find that a high level of symmetry is beneficial for optimization-based methods. On the other hand, presolving 
techniques that detect local substructures to fix variables are crucial for asymmetric instances.
The proposed method outperforms the second best approach, the integer feasibility approach, on 6 of the 12 graphs families and is on par with it on symmetric families.

\end{abstract}
\section{Introduction}

Graphs are versatile representations of complex systems and appear in many applications ranging
from social networks, biology to engineering.
A question of interest in many of these fields is to determine whether two graphs are structurally identical (identical up to vertex ordering), i.e.~whether they are \emph{isomorphic}.
This is called the \emph{Graph Isomorphism problem} (GI) \cite{read1977graph,zemlyachenko1985graph,babai2018group}.
While it is not known whether the problem can be solved in polynomial time, for certain classes of graphs, there are polynomial-time algorithms \cite{hopcroft1972isomorphism,luks1982isomorphism,babai1982isomorphism}.
For an overview and the historic background, we refer the reader to \cite{GI_relaxation}.

A powerful off-the-shelf solver for GI is \texttt{nauty} \cite{mckay2014practical},
originally developed in the 1980s \cite{mckay1981practical}.
The solver computes the automorphism group of a graph which can be utilized to efficiently check if two graphs are isomorphic.
\texttt{nauty} is the state of the art and very fast by, in particular, detecting different graph structures and employing tailored preprocessing algorithms.

In this paper, we investigate optimization formulations and algorithms for the GI problem which can then be embedded within a larger problem.
This work proposes a new optimization-based method to compute isomorphisms or certify that none exists,
leveraging in particular first-order methods similar to \cite{klus2025continuous} to solve an optimization formulation of GI, but instead build an exact method rather than a heuristic.
We utilize the Boscia framework \cite{hendrych2023convex}, a mixed-integer convex optimization framework based on Frank--Wolfe methods, and adapt the solving process to the structure of the GI problem.
In particular, we exploit the formulation of the problem as seeking the minima of a function on the Birkhoff polytope, for which we have access to specialized algorithms for linear assignment.
Additionally, we utilize the detection of local substructures to fix variables as a preprocessing step.
Note that Boscia does not include any preprocessing out-of-the-box as it assumes only oracles access to the objective, its
gradient and the feasible region. 
We compare this approach to the method proposed in \cite{klus2025continuous}, a difference-of-convex algorithm (DCA) approach for which FW methods have recently been studied,
integer feasibility and minimization problems and a heuristic spectral method proposed in \cite{KS18} on a benchmark of selected graph families from the \texttt{nauty} library.
The goal is to computationally assess graph properties that render the problem easy or challenging for the different methods.
High symmetric is beneficial for optimization-based methods as it increases the number of possible solutions.
The proposed variable fixing techniques are crucial for asymmetric graphs as they greatly reduce the size of the 
optimization problem.

The optimization-based approaches are not expected to compete with specialized GI solvers.
However, constrained and inexact variants of the problem appear in many applications, akin to the Quadratic Assignment Problem (QAP).
These can be handled by optimization methods but not by the specialized solvers.
Furthermore, the ideas and formulations we investigate can tackle graphs of up to 500 nodes (250000 variables) despite using generic optimization techniques, showing a high potential for their integration within specialized GI solvers.
While focusing on the GI problem for unweighted undirected graphs, the proposed methods can be naturally extended to weighted and/or directed graphs and to the graph matching problem
in which one seeks the permutation that minimizes the distance between two graphs.
In contrast, \texttt{nauty} does only support directed graphs through graph expansions and weighted graphs have to be presented in an unweighted fashion, greatly increasing problem size.

\subsection*{Notation}

Throughout this work, we use the following notation:
A function $f$ is $L$-smooth if its gradient $\nabla f$ is $L$-Lipschitz continuous over its domain.
We denote the standard inner product as $\innp{\cdot}{\cdot}$,
the Euclidean norm of a vector as $\norm{\cdot}$ and the Frobenius norm of a matrix as $\norm{\cdot}_F$.
Matrices are denoted by uppercase letters, vectors are lowercase bold, scalars are lowercase letters.
The convex hull of a set $\mathcal{X}$ is denoted by $\conv(\mathcal{X})$.
The $n \times n$ identity matrix is denoted as $I_n$, the all-one vector as $\mathbf{1}$.
The set of permutation matrices is denoted by $\mathcal{P}_n$ and its convex hull, the Birkhoff polytope or set of doubly stochastic matrices, is denoted by $\mathcal{D}_n$.

\section{Graph isomorphism problem}
Given two graphs \(G_1 = (V, E_1)\) and \(G_2 = (V, E_2)\) on a common vertex set
\(V\) with \(|V| = n\), the graphs are said to be \emph{isomorphic} if there exists a
bijection \(\pi : V \to V\) such that two vertices \(u, v \in V\) are adjacent in \(G_1\)
if and only if \(\pi(u)\) and \(\pi(v)\) are adjacent in \(G_2\). In other words, an
isomorphism is a relabeling of the vertices that preserves adjacency.

Let \(A\) and \(B\) be the adjacency matrices of \(G_1\) and \(G_2\), respectively.
A vertex bijection corresponds to a permutation matrix \(P \in \mathcal{P}_n\), and
the graphs are isomorphic if and only if $P A P^\top = B$, which can be rewritten as
\(PA = BP\). The GI problem thus consists in
finding a permutation matrix that satisfies this equation. This viewpoint
naturally casts the GI problem as a feasibility problem over the discrete set
\(\mathcal{P}_n\).

To place this feasibility problem within an optimization framework, one may
instead measure the violation of the matching constraint and search for a permutation
that minimizes it. This leads to the quadratic formulation
\[
    \min_{P \in \mathcal{P}_n} \; \| PA - BP \|_F^2,
\]
whose optimal value is zero precisely when the matching condition is satisfied.
The objective is convex in $P$
but the feasible set \(\mathcal{P}_n\) is combinatorial.

Additionally, the presence of graph symmetries may give rise to many distinct permutations
that achieve the same optimal value. These challenges motivate the use of
continuous relaxations, which can provide certificates of non-isomorphism or
serve as effective initializations for discrete methods.

A widely used relaxation replaces \(\mathcal{P}_n\) by its convex hull, the set of doubly stochastic matrices, by the
Birkhoff--von Neumann theorem \cite{Birkhoff_theorem}:
\begin{align*}
    \mathcal{D}_n
    :=
    \left\{
        X \in \mathbb{R}^{n \times n}
        \;\middle|\;
        X\mathbf{1} = \mathbf{1},\;
        X^\top \mathbf{1} = \mathbf{1},\;
        X \ge 0
    \right\}.
\end{align*}
The relaxed problem
\[\min_{X \in \mathcal{D}_n} \; \| X A - B X \|_F^2\]
is therefore a continuous relaxation of the GI problem and can be approached using
convex constrained optimization techniques.

This relaxation is tight for certain classes of graphs, such as asymmetric graphs or
graphs satisfying additional structural conditions \cite{GI_relaxation}, but may admit
fractional optimal solutions for other graph families, which leaves the existence of a permutation with a zero objective undecided.

\section{Integer optimization with Frank--Wolfe}

In this section, we present the method developed for the GI problem based on a first-order method and a branch-and-bound algorithm tailored around it.

\subsection{Frank--Wolfe methods}

The Frank--Wolfe (FW) \cite{FW} or conditional gradient method \cite{CG} is a first-order algorithm for constrained convex optimization.
Given a convex, $L$-smooth function \(f\) and a compact convex feasible region \(\mathcal{X}\),
the method exploits a \emph{linear minimization oracle} (LMO) over the feasible region, which solves
\begin{align*}
\vvv_t \in \arg\min_{\vy \in \mathcal{X}} \langle \vy, \nabla f(\vx_t) \rangle.
\end{align*}
In particular, FW algorithms do not require projections nor algebraic representations of the constraints.
The oracle returns an extreme point of \(\mathcal{X}\) that minimizes the linearization
of \(f\) at the current iterate \(\vx_t\).
The standard FW algorithm is detailed in Algorithm~\ref{alg:fw}.
\begin{wrapfigure}[8]{L}{0.5\textwidth}
\vspace{-1.8\intextsep}
\begin{minipage}[t]{0.5\textwidth}
\begin{algorithm}[H]
\caption{Frank--Wolfe Algorithm}
\label{alg:fw}
\begin{algorithmic}[1]
\State \textbf{Input:} Initial feasible point \( \vx_0 \in \mathcal{X} \), step sizes $\{\gamma_t\}_{t \geq 0}$
\For{$t = 0, \ldots, T$}
\State \( \vvv_t := \argmin_{\vy \in \mathcal{X}} \langle \vy, \nabla f(\vx_t) \rangle \)
\State \( \vx_{t+1} := \vx_t - \gamma_t (\vx_t - \vvv_t) \)
\EndFor
\end{algorithmic}
\end{algorithm}
\end{minipage}
\end{wrapfigure}

The sequence produces feasible iterates without requiring orthogonal projections onto
\(\mathcal{X}\), which can be computationally expensive.
In contrast, the LMO is often significantly cheaper, especially for
polytopes where linear optimization reduces to selecting an appropriate extreme point.
In particular, the Birkhoff polytope admits a specialized implementation of the LMO using the Hungarian algorithm \cite{kuhn1955hungarian,munkres1957algorithms}.
Several methods have enhanced FW in important classes of problems;
we refer readers to the survey \cite{braun2022conditional} for a detailed overview.
We highlight in particular the \emph{Blended Pairwise Conditional Gradient} (BPCG) method \cite{tsuji2021sparser}, which produces sparse iterates by maintaining the current iterate as a sparse convex decomposition of vertices,
and the \emph{Decomposition Invariant Conditional Gradient} (DICG)
\cite{DICG} which removes the need to maintain the active set, by requiring instead an oracle for minimizing functions on particular faces of the feasible region.
BPCG can benefit from so-called \emph{lazification}: one can replace the exact LMO at many iterations by a heuristic providing a vertex with sufficient decrease, see \cite{braun2017lazifying,braun2022conditional},
while DICG performs linear optimization on specific faces of the feasible region.
In particular for the Birkhoff polytope, the Hungarian algorithm can be adapted to efficiently compute the LMO
restricted to a face, which for the Birkhoff polytope corresponds to the same changes performed when branching.
If an entry can be fixed to one, the corresponding row and column can be removed since all other entries have to be zero.
One can also set an entry to zero by removing the edge in the bipartite graph of the linear assignment problem.
Since the dimensions of the problem grow quadratically with the number of nodes, this is a significant reduction
in terms of computational cost compared to utilizing BPCG.
While the worst-case complexity at the root-node remains identical, we highlight that the Hungarian algortithm becomes more tractable when fixed to faces
when using specific data structures \cite{fredman1987fibonacci}.

BPCG has linear convergence under certain conditions, for example under a $1/2-$sharp objective and polyhedral constraints.
The constants of the convergence rate can however result in a dependency on the ambient dimension which could result in a slow rate in practice;
we therefore specify these constants to compute an explicit convergence rate for the convex relaxations we solve depending only on the problem data (i.e., the adjacency matrices).
This rate can then be linked with the cost per iteration dominated by the $\mathcal{O}(n^3)$ Hungarian algorithm.

The original definition of the \emph{pyramidal width} of a polytope is given in \cite{lacoste2015global},
several equivalent definitions were derived in \cite{pena2019polytope} and are typically easier to manipulate and derive for specific polytopes.
We use the facial distance shown to be equal to the pyramidal width \cite{pena2019polytope}.

\begin{definition}[Pyramidal Width \cite{pena2019polytope}]
For any polytope $\mathcal{X}$ with vertex set $\mathcal{V}$, its pyramidal width $\delta$ is equal to the minimum distance of any of its proper faces $F$
to the convex hull of the vertices of $\mathcal{X}$ not lying on $F$. That is:
\begin{equation}
\label{eq:pyramidal_width}
\delta = \min_{F \in \mathrm{faces}(\mathcal{X}): \emptyset \subset F \subset \mathcal{X}} \dist(F, \conv(\mathcal{V} \setminus F)).
\end{equation}
\end{definition}

Computing the pyramidal width is non-trivial.
Thus, we use the observation from \cite{valls2021birkhoff} that the pyramidal width of the Birkhoff polytope is lower bounded by the pyramidal width 
of the $n^2$-dimensional unit cube.
Therefore, we have $\delta \geq 1/n$, see \cite{lacoste2015global} for the proof.
We can now state the linear convergence rate of BPCG.
\begin{proposition}\label{prop:convergence}
Define $f(X)=\norm{XA - BX}_F^2$, $f^*=\min_{X\in\mathcal{D}_n} f(X).$
Then the iterations $\{X_t\}_{t=0}^T$ of the BPCG algorithm guarantee the following:
\begin{equation}
    f(X_t) - f^* \leq 2n(\norm{A}_F + \norm{B}_F)^2 \left( 1 - \frac{1}{4n^3(\norm{A}_F + \norm{B}_F)^2}\right)^{\lceil(t-1)/2\rceil}.
\end{equation}
In particular, if the graphs are undirected, unweighted, and have the same number of edges, $\norm{A}_F = \norm{B}_F = \sqrt{m}$, where $m = |E|$ if the graph is directed and $2|E|$ if the graph is undirected,
leading to:
\begin{equation*}
    f(X_t) - f^* \leq 8mn \left( 1 - \frac{1}{16n^3 m^2 }\right)^{\lceil(t-1)/2\rceil}.
\end{equation*}
\end{proposition}
\begin{proof}
By \cite[Theorem~6]{2023_HendrychBesanconPokutta_Optimalexperimentdesign} and \cite[Corollary~3.33]{braun2022conditional},
we have that
\begin{align*}
    f(X_t) - f^* \leq \left( 1 - \frac{\delta^2}{2M^2LD^2}\right)^{\lceil(t-1)/2\rceil} \frac{LD^2}{2}
\end{align*}
for a $(M,1/2)$-sharp function.
Our objective function $f$ is, in particular, $(1,1/2)$-sharp.
We now show for self-containment of our proof that the diameter of the Birkhoff polytope is $D=\sqrt{2n}$.
This can be verified by considering the maximal distance between two permutation matrices, attained when the sets of $n$ nonzero entries do not intersect.
\begin{align*}
    \norm{X - Y}_F &= \sqrt{\sum_{i=1}^n \sum_{j=1}^n (X_{ij} - Y_{ij})^2} = \sqrt{2n}.
\end{align*}
From the previous discussion, we can set the pyramidal width $\delta=1/n$.
The Lipschitz smoothness constant is \(L=2(\norm{A}_F + \norm{B}_F)^2\).
Combining the different constants into the generic convergence rate, we obtain the required result.
\qed
\end{proof}
\noindent
In the computational experiments, we utilize DICG as it has better empirical performance and memory efficiency.
The different elements of the above proof can be adapted to derive the linear convergence of DICG.

\subsection{The Boscia framework}
Boscia \cite{hendrych2023convex} is a branch-and-bound framework for convex mixed-integer optimization that employs
Frank--Wolfe (FW) algorithms to solve convex relaxations.
The novel aspect of the method is the propagation of the integer constraints to the LMO.
Thus, FW optimizes the objective over the convex hull of integer feasible points, called the \emph{integer hull},
intersected with the node-specific bound constraints.
The LMO always returns an integer feasible point and the B\&B tree obtains an upper bound (incumbent) from the root node.
The \emph{Frank--Wolfe gap} is an upper bound for the primal gap, from the convexity of the objective and the optimality of $\vvv$ for the linear subproblem:
\begin{align*}
    f(\vx) - f^* \leq \max_{\vvv\in\mathcal{D}_n}\innp{\nabla f(\vx)}{\vx - \vvv} =: g(\vx).
\end{align*}
Hence, FW provides a valid lower bound on the optimal value at any iteration.

The Birkhoff polytope is an uni-modular polytope, thus propagating the integer constraints to the LMO does not incur any additional computational cost.

Boscia features a \emph{callback} mechanism, which allows problem-specific logic to be executed after each node is processed and branched on.
Callbacks can inspect intermediate FW iterates, prune nodes prematurely,
terminate the entire algorithm when a desired condition is reached, or inject structural
information back into the solver.
This makes Boscia highly adaptable and particularly
effective for problems where feasibility, symmetry, or combinatorial structure can be
exploited during the tree search.

Within each node, the choice of FW variant plays an important role.
BPCG achieves fast iteration progress at the cost of maintaining an active set.
DICG avoids storing active sets altogether and is well-suited for structured polytopes like the Birkhoff polytope.

We assess two node selection strategies: the best-lower-bound search and the depth-first search.
The former is the most common in branch-and-bound algorithms and selects the node with the smallest lower bound as the next node to process; the latter explores the tree by following a branch until a leaf node is reached, after which it restarts the exploration from the best-bound node.
We also leverage the asymmetry of branching decisions on the Birkhoff polytope, always branching up (i.e., rounding to one) the branching variable during the depth-first search, thus building search trees of depth at most $n$.

\subsection{Variable fixing through local graph structures and perturbations}

Problem presolving is an essential driver of optimization methods efficiency.
Unlike specialized GI software or mixed-integer programming (MIP) solvers, Boscia does not perform presolving techniques because it only assumes an oracle access to both the objective and constraints.
In our GI formulation optimizing over the Birkhoff polytope through the Hungarian algorithm, adding general cutting planes would break the constraint structure and require a generic solver for the LMO.
Fixing variables of the permutation matrix to zero or one can however be performed, and results in an LMO consisting of a modified and less expensive call to the Hungarian algorithm.
We propose two presolving techniques detecting local graph structures and one optimization-based fixing technique.

The local graph structure consists in two steps: quantifying a property of each vertex that is independent of the vertex ordering,
and fixing to zero the entry $X_{ij}$ if $i$ and $j$ do not agree on that property.
The two properties we identify are the number of cliques of size $3\dots k$ that contain the current vertex.
This can be performed in polynomial time for a fixed $k$.
When the graph has a bounded degree, we can even use the Bron-Kerbosch algorithm \cite{eppstein2013listing} to compute all maximal cliques
and thus derive the number of maximal cliques of each size which each given vertex belongs to.
A pair of vertices from the first and second graphs can correspond to one another only if they have the same clique count for each clique size.
As a second substructure, we count maximal stars on the subgraphs induced by the neighborhood of each vertex, which correspond to independent sets on that graph.
Using a similar reasoning, if the graph has bounded degree, each induced subgraph has a fixed maximum size and we can compute all maximal independent sets.
Otherwise, we can compute independent sets up to a fixed size in these induced subgraphs.

As a last presolving step, we propose an optimization-based technique that tentatively fixes entries in the permutation matrix.
It exploits the non-negativity of the objective and only requires calls to the Hungarian algorithm and gradient computations.
For a given entry $(i,j)$ of the permutation matrix, we run a FW method with a limited iteration budget on the modified problem
\begin{align*}
\min_{X \in \mathcal{D}_n} \|X A - B X\|^2 + X_{ij}.
\end{align*}
If at an iteration $t$, the dual bound given by the primal value minus the FW gap is above zero,
the optimal value cannot attain zero.
From nonnegativity of the two objective parts, this implies that $X_{ij} > 0$ in any solution, and thus that $X_{ij}$ can be fixed to one, eliminating a row and column.
We can equivalently optimize with a term $1 - X_{ij}$ to attempt to fix $X_{ij}$ at zero.
This technique can be viewed as a modified optimization-based bound tightening (OBBT) using an augmented Lagrangian formulation, relating to earlier work on generic mixed-integer optimization \cite{gleixner2017three}.
Furthermore, a set of variable fixings that is infeasible on the Birkhoff polytope provides a proof of non-isomorphism of the graphs.
This technique has a polynomial runtime but is much costlier in practice than the substructure presolve.

\section{Solution approaches}\label{sec:solution-approaches}

First, we consider the exact integer formulation
\[
    \min_{X \in \mathcal{P}_n} \, \|XA - BX\|_F^2,
\]
which we solve using \textbf{Boscia}. 
At each node,
Boscia optimizes the objective to a given tolerance over the corresponding integer hull using DICG.
This variant consistently exhibited superior performance both in terms of progress per iteration and wall-clock time
compared to standard FW and BPCG.
The LMO is computed with the Hungarian algorithm which, especially for large dimensions, is computationally more 
efficient than calling an LP solver.
The callback mechanism provides effective control over the search. After a node is
processed, its lower bound is inspected; if the bound is strictly positive, the subtree
cannot contain a feasible solution and is pruned. The callback also checks the incumbent and the lower bound of the tree;
if the incumbent reaches zero, a
permutation matrix satisfying \(XA = BX\) has been found, and the algorithm terminates,
certifying isomorphism.
On the other hand, if the lower bound is greater than zero, it certifies non-isomorphism.

As discussed earlier, the convex relaxation of the GI problem may admit fractional
solutions that do not correspond to valid isomorphisms. One approach to address this
issue is to augment the formulation with a penalty term that promotes solutions closer to
permutation matrices, as proposed in~\cite{klus2025continuous}. The resulting optimization problem is
\[
    \min_{\substack{X \in \mathcal{D}_n \\ XA = BX}} \; -\|X\|_F^{2},
\]
where the constraint $XA = BX$ enforces feasibility with respect to the isomorphism
condition, and the objective $-\|X\|_F^{2}$ biases the solution toward the extreme points
of~$\mathcal{D}_n$, i.e., the permutation matrices.
We denote it by \textbf{Penalty}.
The above problem is concave, hence, there is no guarantee that Frank--Wolfe will converge to the global minimum.
This may be mitigated by generating a good initial point, see~\cite{klus2025continuous} for more details.
Note that~\cite{klus2025continuous} uses an LP solver to generate the initial point from the convex relaxation.
The relaxation is naturally suited to Frank--Wolfe methods, which can produce a feasible fractional point far more efficiently.
Nevertheless, this method can only serve as a heuristic.

A drawback of the previous formulation is the constraint $XA = BX$ as it leads to expensive LP subproblems.
The difference-of-convex (DC) formulation
\[
    \min_{X \in \mathcal{D}_n}
        \; \| XA - BX \|_F^{2}
        - \lambda \|X\|_F^{2},
\]
moves the constraint into the objective.
Thus, the LMO can be computed with the Hungarian algorithm.
Like the aforementioned penalty formulation, the problem is nonconvex but its structure
allows the use of FW-type algorithm for DC, proposed in~\cite{maskan2025revisiting} and extended in~\cite{DCA}.
It will be further identified by \textbf{DC-FW}.
This approach cannot guarantee finding an optimal permutation, but it may lead to faster convergence compared to running Frank--Wolfe directly.
While DCA guarantees convergence to a
stationary point, it offers no guarantee of global optimality. The quality of the
solution is sensitive to the choice of the parameter \( \lambda \); for instance, a
large \( \lambda \) may steer the iterates away from matrices satisfying
\( \| XA - BX \|_F = 0 \). 
In our experiments, we have set $\lambda=10^{-2}$.

The GI problem can also be formulated as a pure binary feasibility problem.
Given adjacency matrices \(A\) and \(B\), the graphs are isomorphic if and only if there exists
a permutation matrix \(X\) such that
\[
    XA = BX, \; X \in \mathcal{P}_n.
\]
where \(\mathcal{P}_n\) is the set of all permutation matrices.
This formulation can be handed directly to a \textbf{MIP} solver, in this case \textsc{SCIP}, by enforcing the linear and integrality
constraints defining \(\mathcal{P}_n\) together with the equation \(XA = BX\).
Although potentially expensive for large instances, this provides an exact and simple
baseline against which we compare our other approaches.
Furthermore, to the best of our knowledge, there are no assessement of the practical performance of a recent MIP solver on the GI problem.

A \textbf{Spectral} assignment approach for detecting isomorphisms was proposed in~\cite{KS18}.
The method relies on iteratively perturbing the adjacency matrices of the two graphs to break symmetries and assigning vertices of $ G_1 $ to vertices of $ G_2 $ by solving linear assignment problems that are based on the eigenvalues and eigenvectors of the perturbed matrices.
While it can be shown that the graph isomorphism problem can be solved in polynomial time if the graphs are friendly (see, e.g., \cite{ABK15}), the method might require backtracking for graphs with repeated eigenvalues and therefore does not have a polynomial runtime.

\section{Computational experiments}

Our benchmark consists of 12 graph families from the \texttt{nauty} benchmark library \cite{nauty-traces-benchmarks}, ranging from 10 to 500 nodes.
The \texttt{exact} set is from the PACE 2023 challenge \cite{pace-2023-twinwidth}.
All graphs except the \texttt{exact} set are regular.
For each graph, we generate three isomorphic instances by applying
random permutation matrices. An instance is counted as solved if the method correctly
identifies isomorphism within the time limit of 1 hour.
We compare methods by the number of solved instances and their solving times.

The methods described in Section~\ref{sec:solution-approaches} are implemented in \texttt{Julia} v1.10.3.
We use \textsc{SCIP} 9.2.4 \cite{bestuzheva2023enabling,BolusaniEtal2024OO} as the MIP solver through its \textsc{MathOptInterface.jl} wrapper.
The integer Frank--Wolfe algorithms rely on \textsc{Boscia.jl} v0.2.3~\cite{hendrych2023convex} and
\textsc{FrankWolfe.jl} v0.6.1~\cite{besanccon2022frankwolfe,besanccon2025improved}.
The linear assignment oracle uses \textsc{Hungarian.jl} v0.7.
All experiments are conducted on a 32-core node (Intel(R) Xeon(R) Gold 6338, 2.00GHz, 1024GB RAM).
We utilize the \emph{secant} line search in FW variants, which amounts to exact line search with a single additional gradient call on quadratic functions \cite{hendrych2025secant}.

The families are split by symmetry in the tables where we consider families with normalized orbits of 0.0--0.3 as highly symmetric (top), 0.3-0.7 as varying degrees of symmetry (middle) and 0.7--1.0 as low symmetry (bottom).
Table~\ref{tab:gi_results_summary} shows the performance of the selected methods on the benchmark.
Note that both penalty heuristics as well as the spectral method were always dominated by either one of the Boscia versions
and/or the MIP approach, so we do not report the results in the table.
Their performance is shown for selected graph families in later figures.
As highlighted in the Section~\ref{sec:solution-approaches}, the spectral method potentially has to backtrack to find an
isomorphism.
Due to the nonconvexity of the penalty-based heuristics, they may converge to a local minimum.
The first four columns of the table denote different versions of Boscia.
The first version is Boscia without any preprocessing and the depth-first-search (DFS) traverse strategy.
The next three utilize different preprocessing steps in Boscia: star detection, clique and star detection and the OBBT fixings.
All of these versions use the DFS traverse strategy.
Note that the presolving methods of \textsc{SCIP} are enabled and it calls \texttt{nauty} as part of its preprocessing. 
We also conducted additional experiments with the $\ell_1$-vector-norm formulation solved via MIP solver, which consistently underperformed the
feasibility formulation proposed in the paper, thus it is not reported separately.

Table~\ref{tab:gi_results_summary} also shows the results for \texttt{nauty} for completeness. 
Yet we highlight that we do not compare the proposed methods to \texttt{nauty} as it is specialized for GI
and not an optimization method.

\begin{table}[h]
    \centering
    \Large
    \renewcommand{\arraystretch}{1.40}
    \resizebox{\textwidth}{!}{%
    \begin{tabular}{lr ccc ccc ccc ccc ccc | ccc}
    \toprule
    Family & Inst. & \multicolumn{3}{c}{Boscia DFS} & \multicolumn{3}{c}{Boscia Star} & \multicolumn{3}{c}{Boscia Clique \& Star} & \multicolumn{3}{c}{Boscia Fixings} & \multicolumn{3}{c}{MIP} & \multicolumn{3}{c}{Nauty} \\
    \cmidrule(lr){3-5} \cmidrule(lr){6-8} \cmidrule(lr){9-11} \cmidrule(lr){12-14} \cmidrule(lr){15-17} \cmidrule(lr){18-20}
    & & \# & \% & Time (s) & \# & \% & Time (s) & \# & \% & Time (s) & \# & \% & Time (s) & \# & \% & Time (s) & \# & \% & Time (s) \\
    \midrule
    
    latin & 63 & 34 & \textbf{54\%} & 289.71 & 34 & \textbf{54\%} & 298.79 & 34 & \textbf{54\%} & 303.30 & 15 & 24\% & 1530.52 & 34 & \textbf{54\%} & 257.23 & 63 & 100\% & 0.38 \\
    Lattice & 21 & 15 & 71\% & 201.26 & 15 & 71\% & 202.80 & 13 & 62\% & 218.84 & 6 & 29\% & 1621.49 & 18 & \textbf{86\%} & 96.07 & 21 & 100\% & 0.37 \\
    paley\_power & 24 & 19 & \textbf{79\%} & 109.66 & 19 & \textbf{79\%} & 124.95 & 19 & \textbf{79\%} & 122.61 & 6 & 25\% & 1282.27 & 10 & 42\% & 365.73 & 24 & 100\% & 0.37 \\
    paley\_prime & 21 & 18 & \textbf{86\%} & 46.26 & 18 & \textbf{86\%} & 51.45 & 18 & \textbf{86\%} & 67.53 & 6 & 29\% & 1316.20 & 10 & 48\% & 234.53 & 21 & 100\% & 0.35 \\
    Triangular & 21 & 15 & 71\% & 241.13 & 15 & 71\% & 256.89 & 15 & 71\% & 254.69 & 6 & 29\% & 2189.35 & 17 & \textbf{81\%} & 122.98 & 21 & 100\% & 0.35 \\
    \midrule
    CHH\_cc & 24 & 8 & 33\% & 740.13 & 8 & 33\% & 870.83 & 8 & 33\% & 842.36 & 6 & 25\% & 2150.97 & 22 & \textbf{92\%} & 80.67 & 24 & 100\% & 0.40 \\
    tnn & 42 & 13 & 31\% & 1299.91 & 14 & 33\% & 887.11 & 14 & 33\% & 919.52 & 6 & 14\% & 2261.95 & 30 & \textbf{71\%} & 106.28 & 35 & 83\% & 14.55 \\
    \midrule
    cfi & 12 & 0 & 0\% & 3600.00 & 0 & 0\% & 3600.00 & 0 & 0\% & 3600.00 & 0 & 0\% & 3600.00 & 11 & \textbf{92\%} & 383.75 & 12 & 100\% & 0.57 \\
    exact & 21 & 21 & \textbf{100\%} & 35.57 & 21 & \textbf{100\%} & 5.21 & 18 & 86\% & 22.75 & 6 & 29\% & 1302.23 & 12 & 57\% & 79.76 & 21 & 100\% & 0.36 \\
    iso\_r01N & 21 & 18 & \textbf{86\%} & 35.14 & 18 & \textbf{86\%} & 5.91 & 18 & \textbf{86\%} & 6.26 & 9 & 43\% & 1025.71 & 18 & \textbf{86\%} & 32.33 & 21 & 100\% & 0.35 \\
    sts & 21 & 0 & 0\% & 3600.00 & 12 & 57\% & 444.28 & 18 & \textbf{86\%} & 411.44 & 0 & 0\% & 3600.00 & 0 & 0\% & 3600.00 & 21 & 100\% & 1.10 \\
    usr & 36 & 5 & 14\% & 2538.99 & 30 & \textbf{83\%} & 58.27 & 18 & 50\% & 200.90 & 4 & 11\% & 3164.43 & 6 & 17\% & 1867.12 & 30 & 83\% & 15.94 \\
    \bottomrule
    \end{tabular}%
    }
    \caption{Performance summary by graph family. Best performance for each family is highlighted in bold.
    The reported times are the geometric mean of the total solving times shifted by 1 second for all instances. This includes timed out instances.}
    \label{tab:gi_results_summary}
\end{table}

High symmetry seems to have an overall positive effect. 
The reason is that the number of possible solution increases with the symmetry. 
The relative weakness of the effect is likely due to the graphs in the symmetric families being rather dense.

The multitude of solutions in the symmetric case also explains why the preprocessing steps do not improve
the performance of Boscia.
On the other hand, it has a big positive impact on the asymmetric instances.
Especially on the \texttt{usr} set, Boscia with preprocessing outperforms the MIP approach.
We further highlight that even \texttt{nauty} is struggling on this set compared to its performance on the other sets.

For the middle level symmetric families, all of the Boscia variants are outperformed by the MIP approach.
They seem to exhibit enough symmetry that the preprocessing is not very effective but do not admit too many possible solutions.
Nonetheless, the MIP feasibility approach is beating our method on only four families out of twelve.

Other exact approaches for GI are based on the NP-hard Quadratic Assignment Problem (QAP) \cite{aurora2018qap}, using in 
particular a lifted formulation resulting in $\mathcal{O}(n^4)$ variables.
We did not test it given its high computational cost for moderate graphs compared to the compact MIP feasibility problem, also reported by the authors.

The OBBT fixing, while effective for low-symmetry graphs, is in its current implementation far too costly, see Table~\ref{tab:fixings_analysis}.
Thus, we solve hardly more instances than the baseline DFS version. 
For symmetric graphs in particular, OBBT hardly fixes any entries. 
Nonetheless, the design of an efficient OBBT-based algorithm is a promising avenue for future research.
\begin{table}[h]
    \centering
    \normalsize
    \renewcommand{\arraystretch}{1.10}
    \resizebox{\textwidth}{!}{%
    \begin{tabular}{lrrrrrr}
    \toprule
    Family & \# Graphs & \% Solved & Time (s) & Fixing Time (s) & Fixings/$n^{\scalebox{0.6}{2}}$ & Avg Iters \\
    \midrule
    
    latin & 63 & 23.8\% & 1530.52 & 1523.19 & 0.16\% & 23 \\
    Lattice & 21 & 28.6\% & 1621.49 & 1614.70 & 0.00\% & 0 \\
    paley\_power & 24 & 25.0\% & 1282.27 & 1275.75 & 0.00\% & 0 \\
    paley\_prime & 21 & 28.6\% & 1316.20 & 1310.75 & 0.00\% & 0 \\
    Triangular & 21 & 28.6\% & 2189.35 & 2182.52 & 0.00\% & 0 \\
    \midrule
    CHH\_cc & 24 & 25.0\% & 2150.97 & 2145.75 & 64.25\% & 33 \\
    tnn & 42 & 14.3\% & 2261.95 & 2259.73 & 77.44\% & 13 \\
    \midrule
    cfi & 12 & 0.0\% & 3600.00 & 3600.00 & 0.00\% & 0 \\
    exact & 21 & 28.6\% & 1302.23 & 1299.37 & 59.93\% & 11 \\
    iso\_r01N & 21 & 42.9\% & 1025.71 & 1022.23 & 30.60\% & 16 \\
    sts & 21 & 0.0\% & 3600.00 & 3600.00 & 0.00\% & 0 \\
    usr & 36 & 11.1\% & 3164.43 & 2818.93 & 0.00\% & 0 \\
    \bottomrule
    \end{tabular}%
    }
    \caption{Fixings analysis by graph family. \# Graphs: number of instances; \% Solved: percentage solved within time limit; Time (s): geometric mean of total solving time shifted by 1 second; Fixing Time (s) geometric mean of time spent on optimization-based bound tightening (OBBT) fixings shifted by 1 second; Fixings/$n^{2}$: percentage of permutation matrix entries fixed relative to $n^2$; Avg Iters: average number of Frank--Wolfe iterations for the fixing procedure.}
    \label{tab:fixings_analysis}
\end{table}

Clique and star detection likewise do not have much impact for symmetric graphs.
In contrast to the fixings, they are very cheap and thus also do not negatively effect the performance.
On the asymmetric instances, the proposed method benefits a lot from clique and star detection.
Notable exception is the \texttt{exact} set. 
The reason is that it is the only set which does, in fact, include non-regular graphs.
The assumption of regularity is crucial for the clique and star detection to be effective.

\begin{figure}[t]
    \centering
    \subfloat[Comparing all methods]{\includegraphics[width=0.65\textwidth]{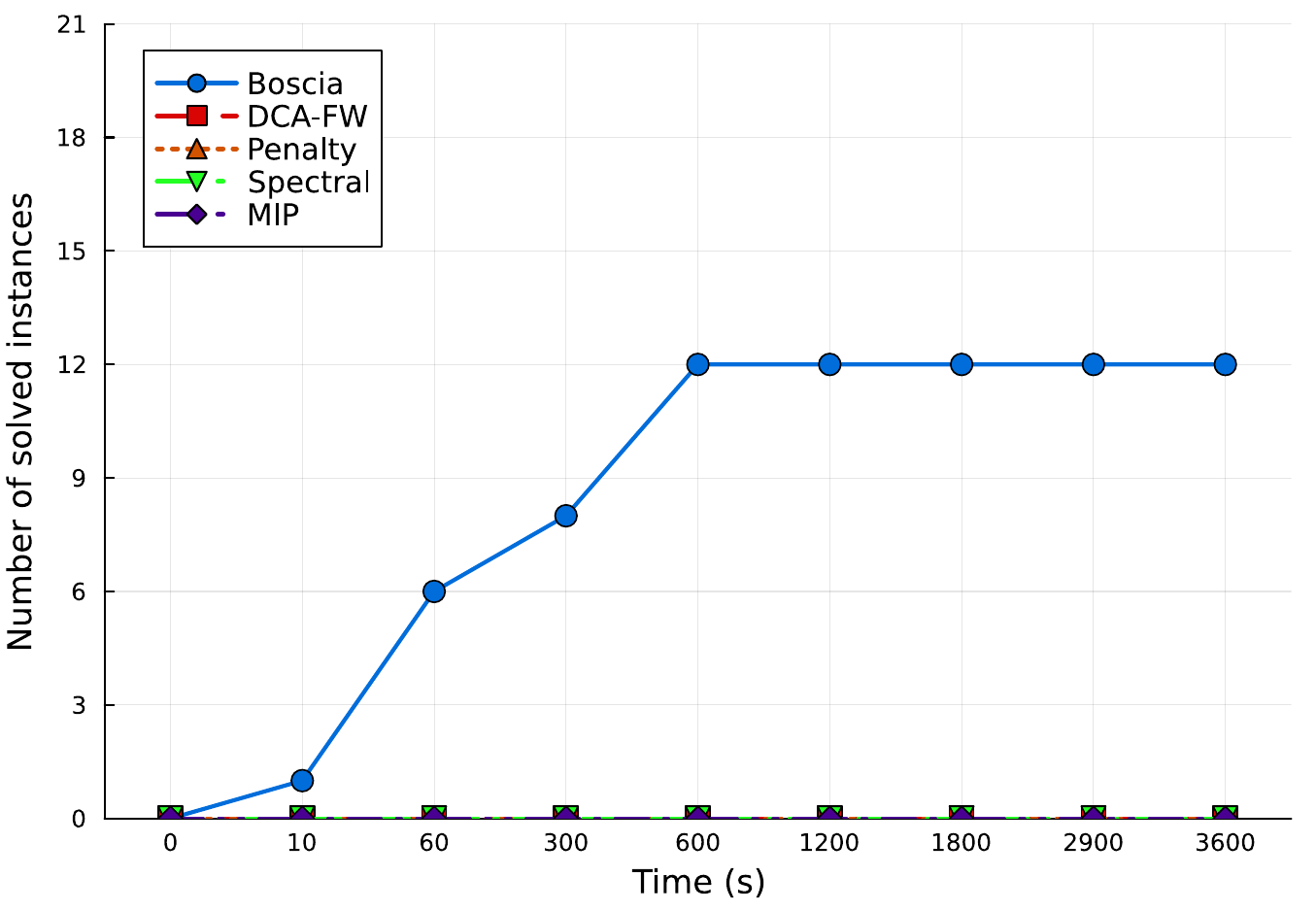}}
    \hfill
    \subfloat[Comparing Boscia variants\label{fig:sts_boscia_summary}]{\includegraphics[width=0.65\textwidth]{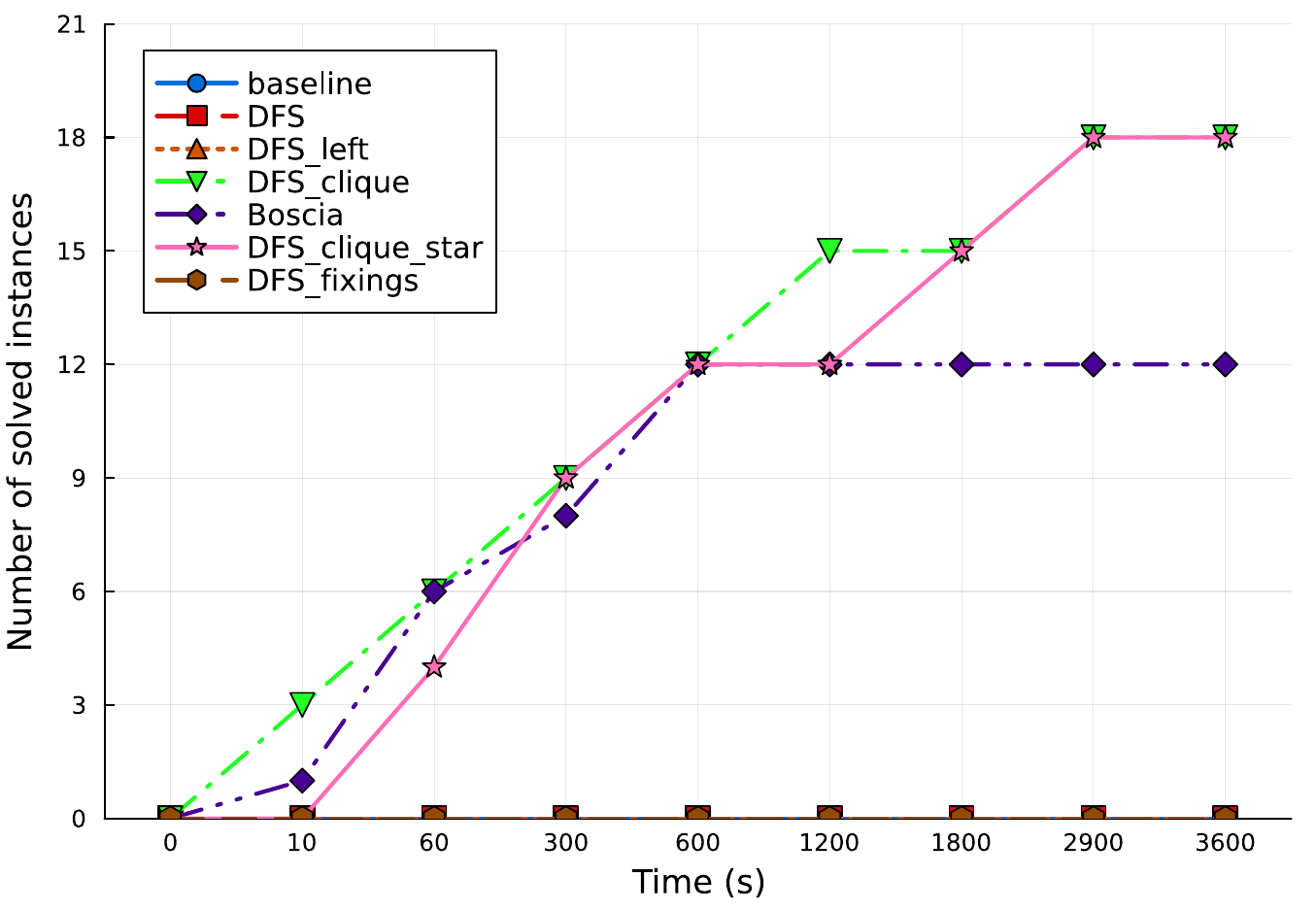}}
    \caption{Solved instances over time for \texttt{sts} graph family.}
    \label{fig:sts_summary}
\end{figure}

In Figures \ref{fig:sts_summary} to \ref{fig:iso_r01N_summary}, we showcase the number of solved instances
over time for the different methods on a selection of families.
On the left side, we compare the different methods. 
Boscia with DFS and the star detection has overall the best performance of all variants and it is thus the chosen 
method for the comparison plots, denoted just as Boscia.
On right, we compare the different versions of Boscia.
The baseline is standard Boscia with best-lower-bound search and without any preprocessing.
By Boscia, we denote Boscia with DFS and the star detection enabled like on the left side.
The DFS\_left variant is the depth-first-search but favoring the left child, so fixing an entry to zero. 
Given the structure of the Birkhoff polytope, this is predictably not as performant as favouring the right child.

On the \texttt{sts} family, shown in Figure~\ref{fig:sts_summary}, only our method is able to solve a significant amount of instances,
neither of the other approaches solved any instance.
The graphs in this family are sparse but exhibit low symmetry.
The penalty and spectral methods generally benefit from symmetry and struggle on more asymmetric instances.
We also observe from Figure~\ref{fig:sts_boscia_summary} that clique detection makes a significant performance difference. 

\begin{figure}[t]
    \centering
    \subfloat[Comparing all methods]{\includegraphics[width=0.65\textwidth]{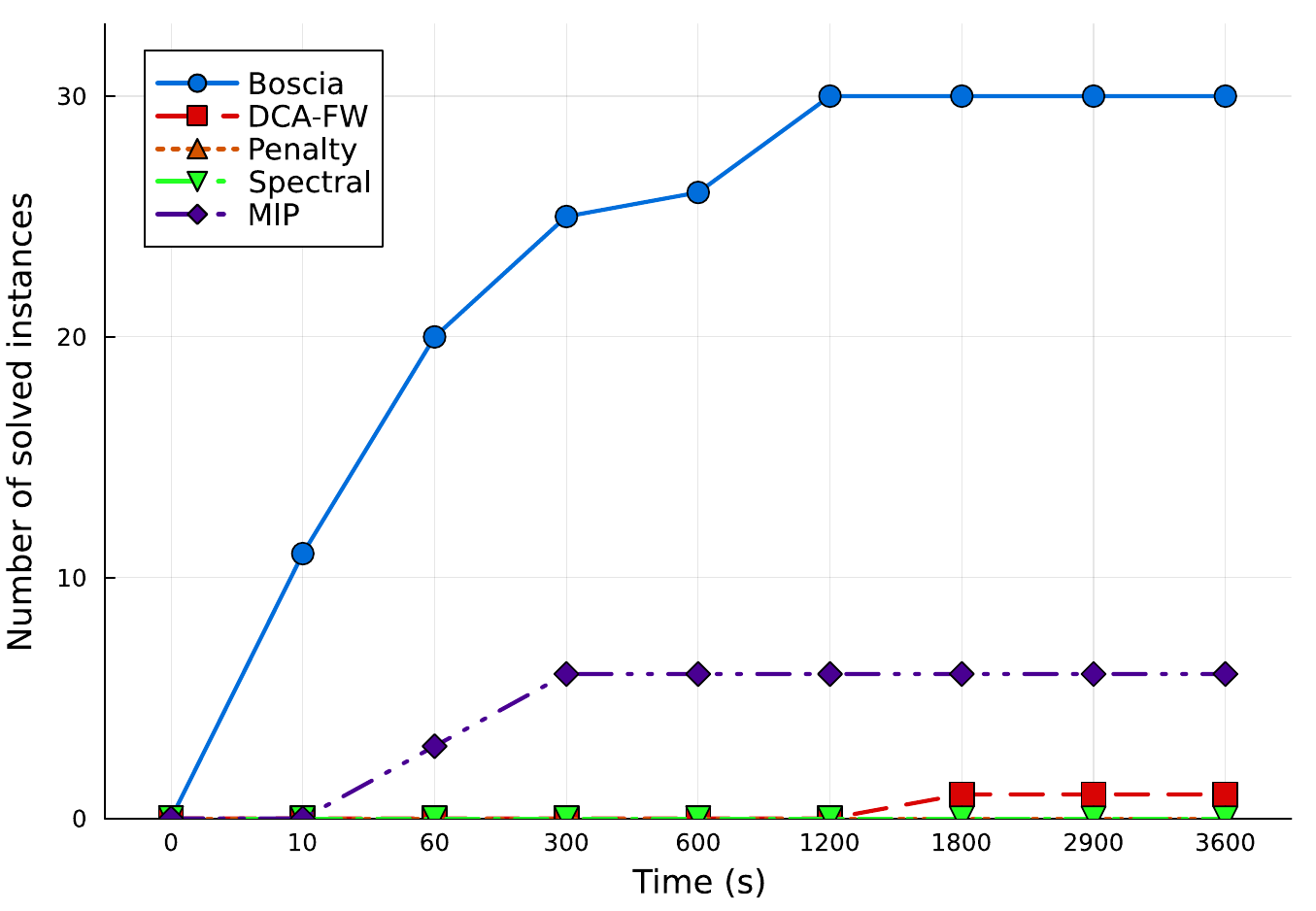}}
    \hfill
    \subfloat[Comparing Boscia variants\label{fig:usr_boscia_summary}]{\includegraphics[width=0.65\textwidth]{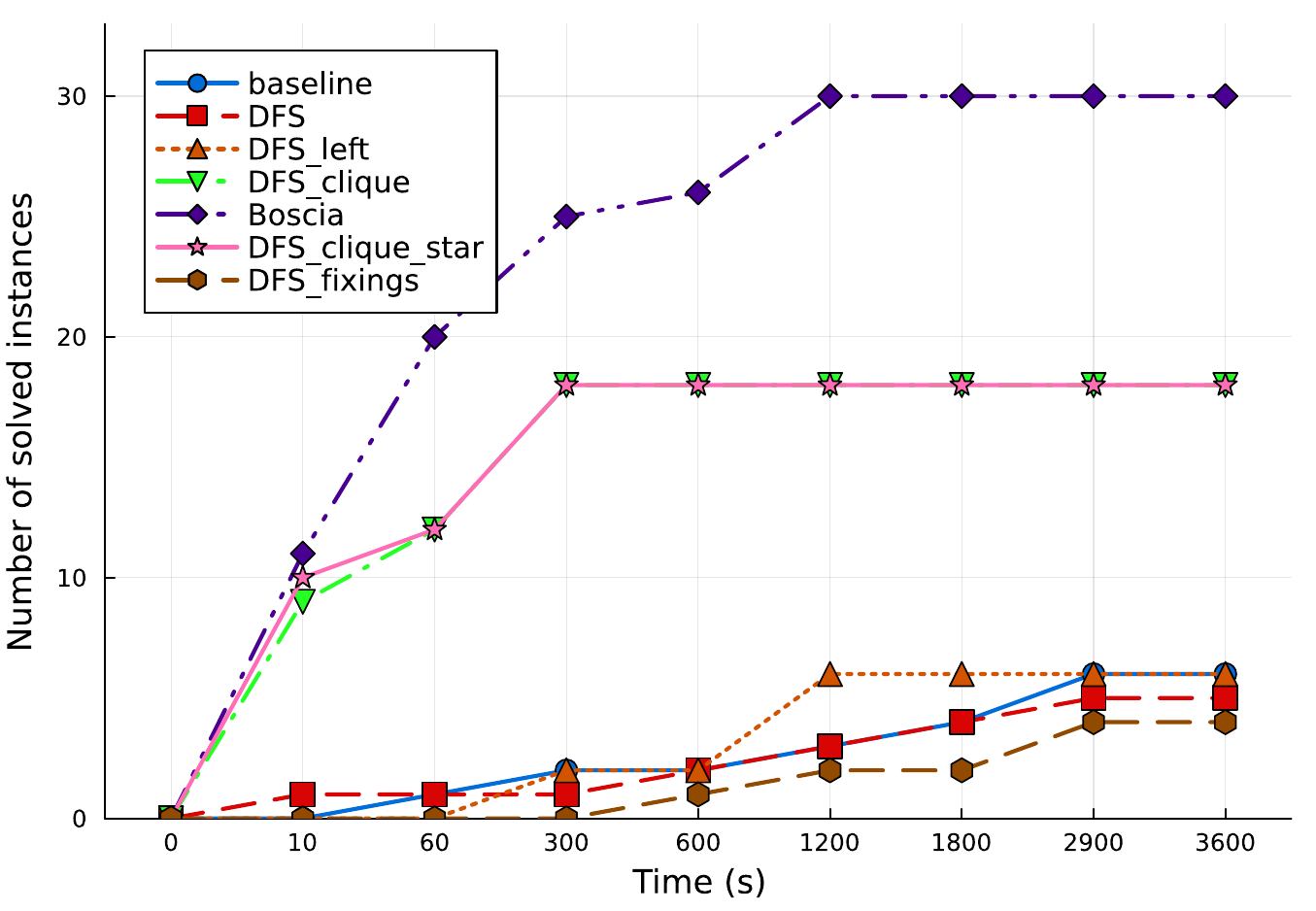}}
    \caption{Solved instances over time for \texttt{usr} graph family.}
    \label{fig:usr_summary}
\end{figure}

The same behavior can be observed for the \texttt{usr} family on Figure~\ref{fig:usr_summary}.
The MIP performs better here than on the \texttt{sts} family, but is not competitive.
This graph family is the densest one in our benchmark, so this could account for performance of the MIP approach.

\begin{figure}[t]
    \centering
    \subfloat[Comparing all methods]{\includegraphics[width=0.65\textwidth]{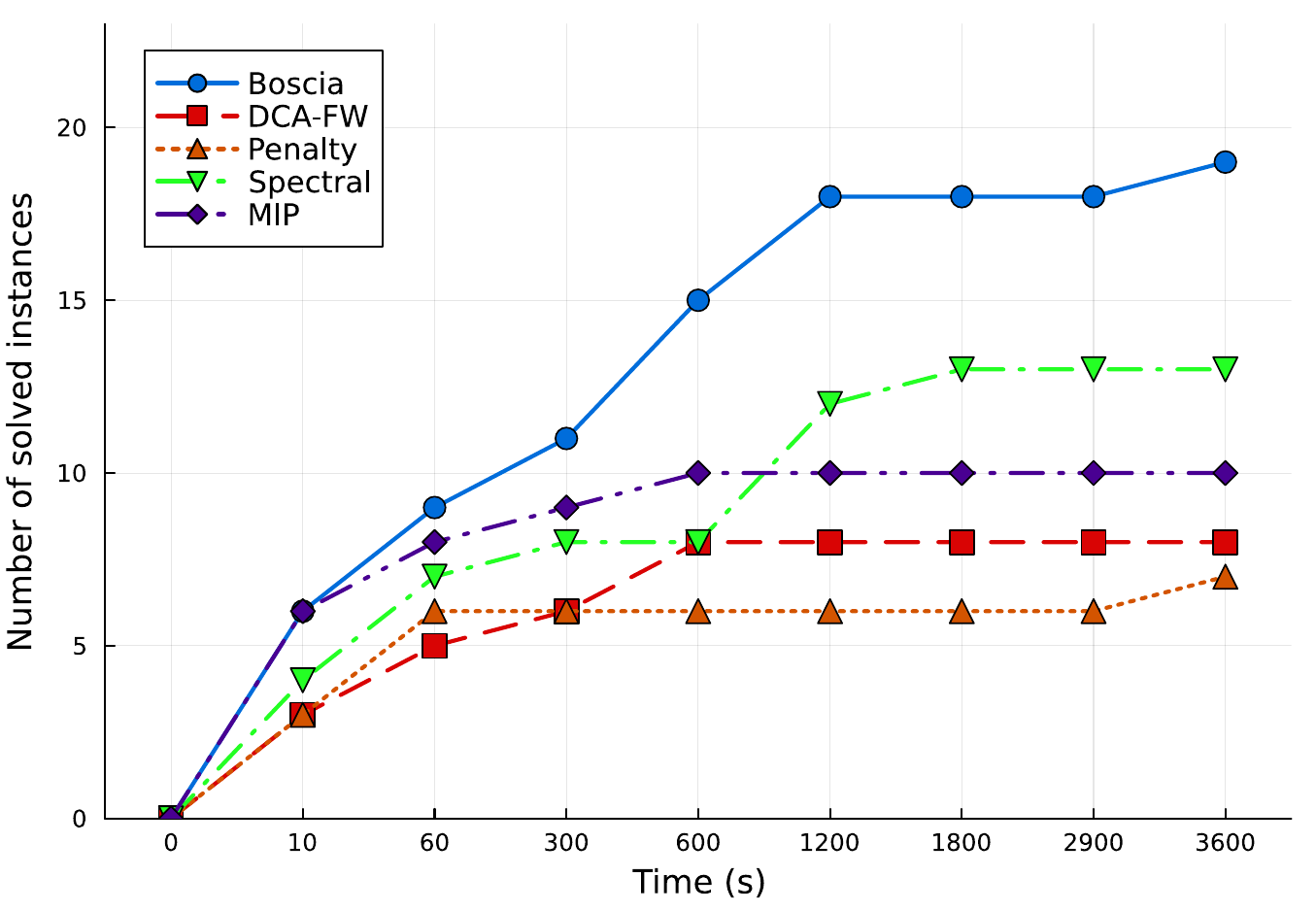}}
    \hfill
    \subfloat[Comparing Boscia variants]{\includegraphics[width=0.65\textwidth]{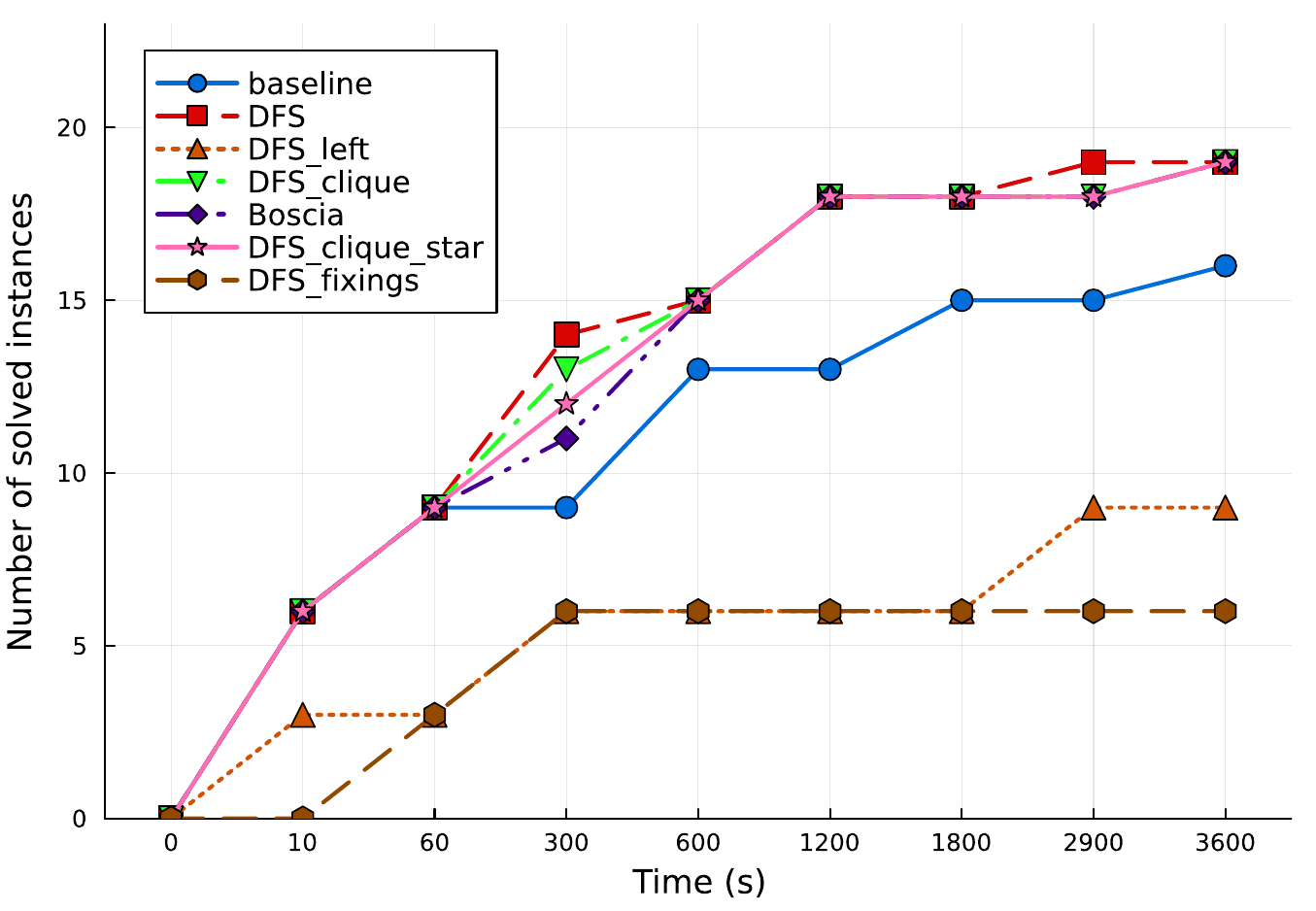}}
    \caption{Solved instances over time for \texttt{paley\_power} graph family.}
    \label{fig:paley_power_summary}
\end{figure}

\begin{figure}[t]
    \centering
    \subfloat[Comparing all methods]{\includegraphics[width=0.65\textwidth]{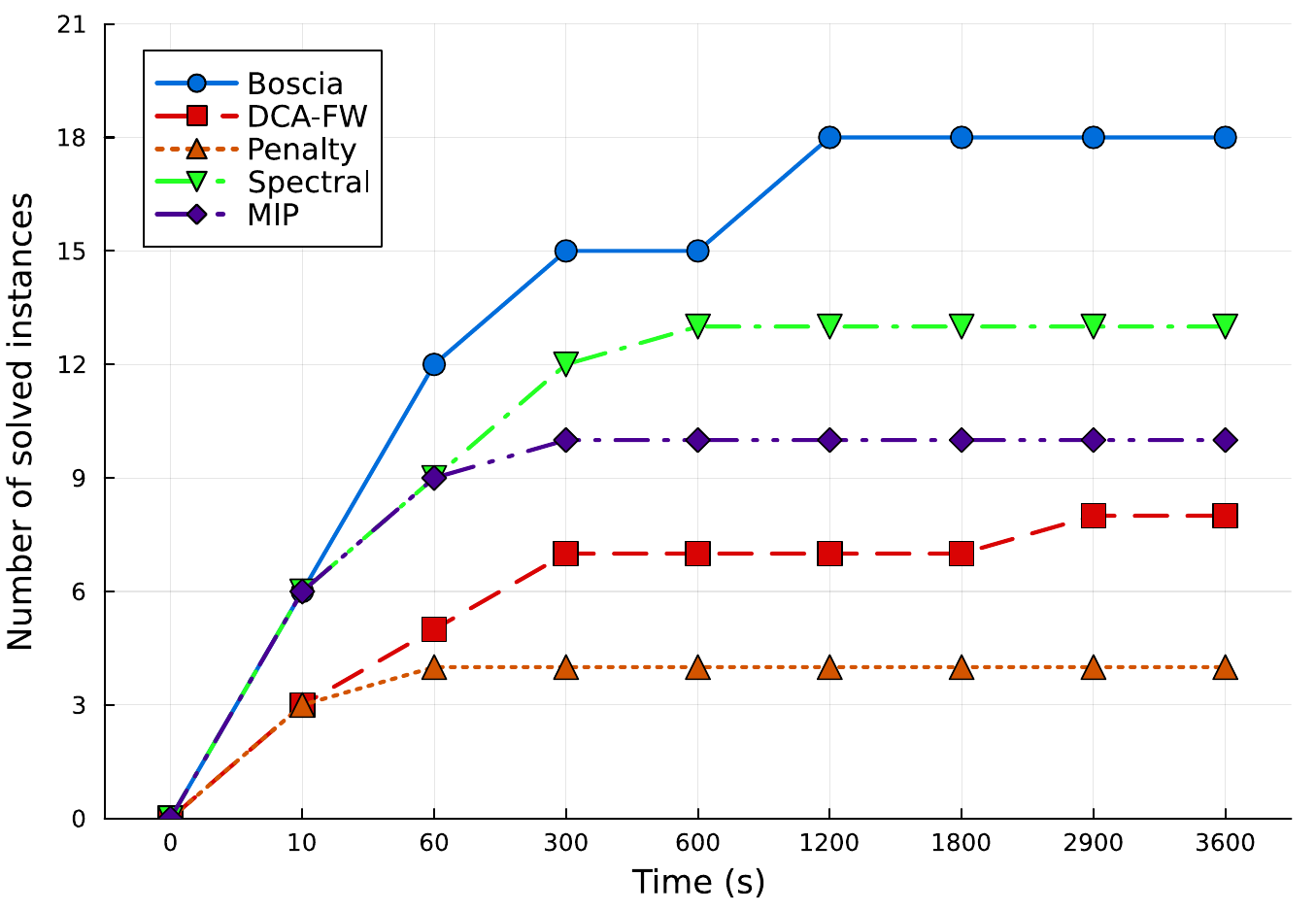}}
    \hfill
    \subfloat[Comparing Boscia variants]{\includegraphics[width=0.65\textwidth]{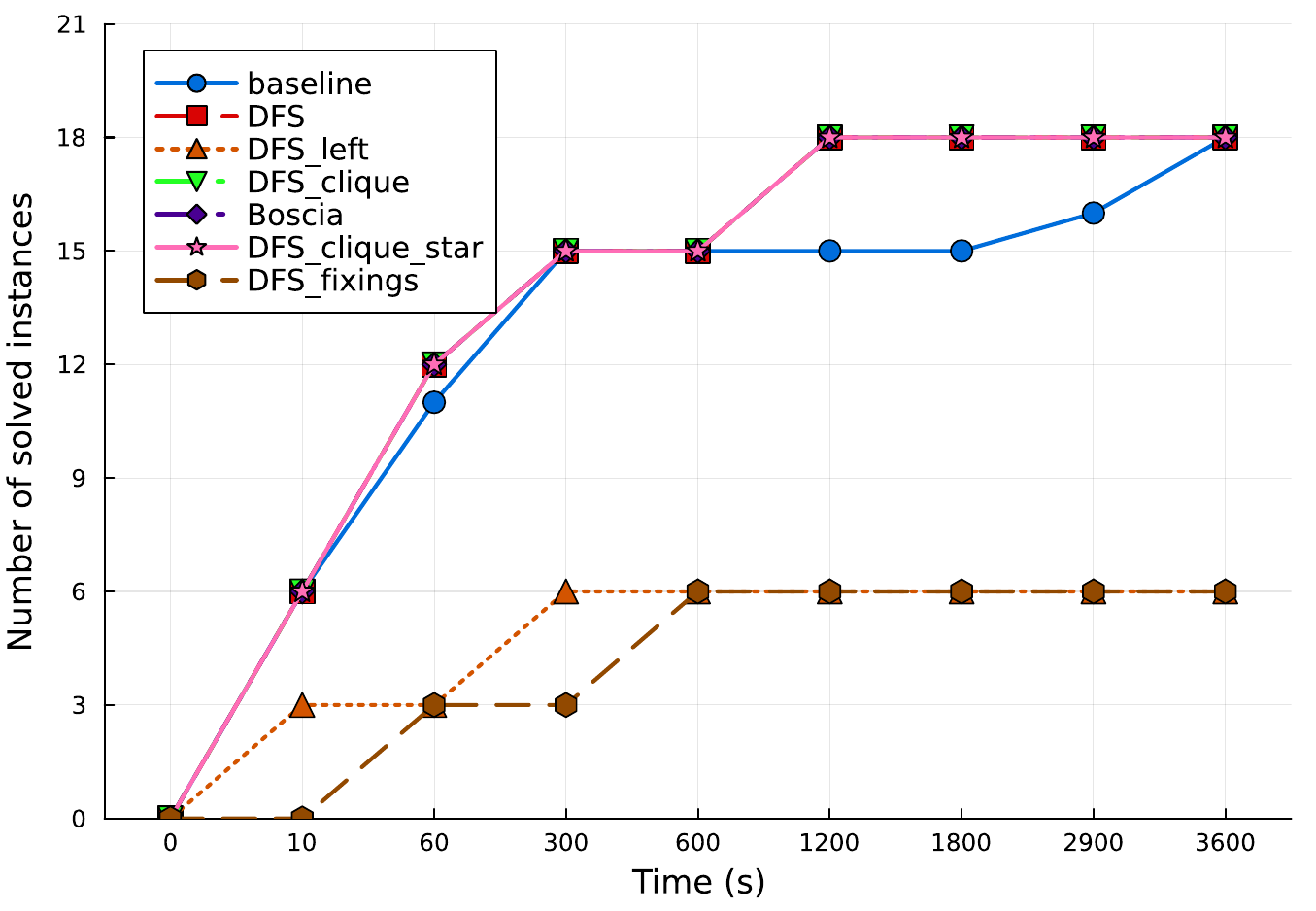}}
    \caption{Solved instances over time for \texttt{paley\_prime} graph family.}
    \label{fig:paley_prime_summary}
\end{figure}

For both \texttt{Paley} graph families, Figures \ref{fig:paley_power_summary} and \ref{fig:paley_prime_summary},
the clique and star detection are nearly equally effective.
Our method outperforms all other methods on these sets as well.
Interestingly, the spectral method outperforms the feasibility approach on these sets.
It, in particular, fares well on the \texttt{paley\_prime} family.

\begin{figure}[t]
    \centering
    \subfloat[Comparing all methods]{\includegraphics[width=0.65\textwidth]{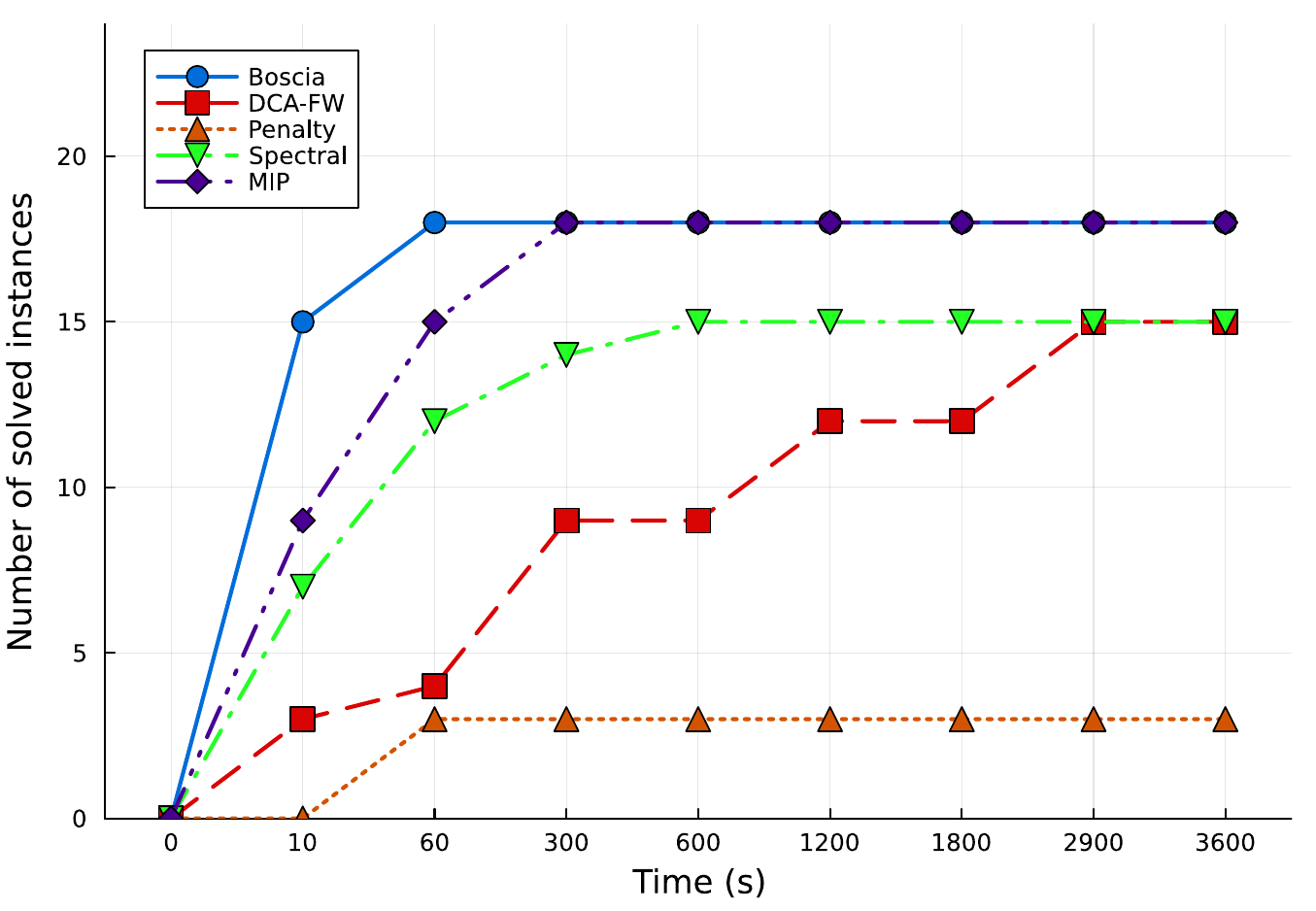}}
    \hfill
    \subfloat[Comparing Boscia variants]{\includegraphics[width=0.65\textwidth]{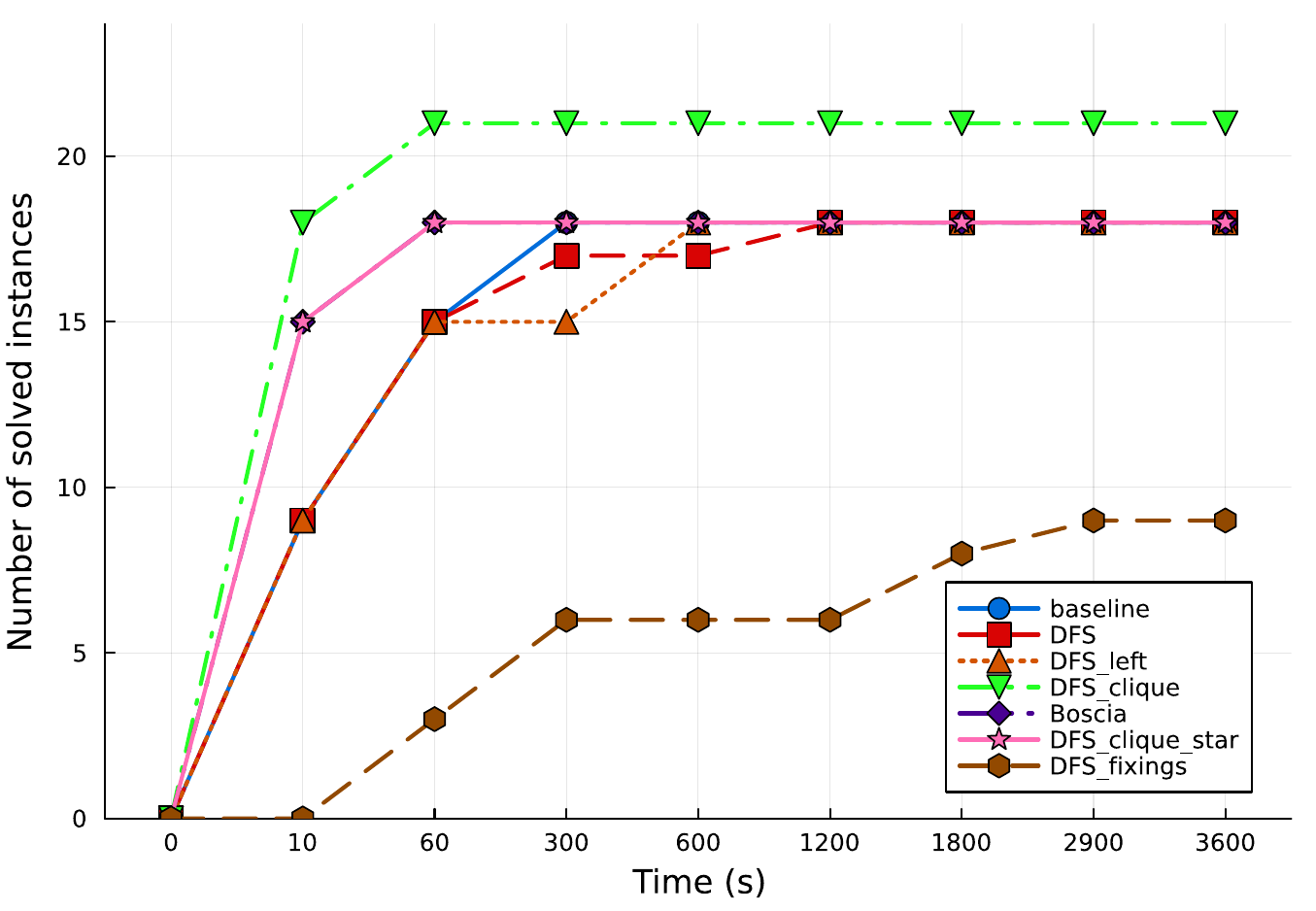}}
    \caption{Solved instances over time for \texttt{iso\_r01N} graph family.}
    \label{fig:iso_r01N_summary}
\end{figure}

The \texttt{iso\_r01N} family, Figure \ref{fig:iso_r01N_summary}, is relatively easy to solve for most methods.
This is somewhat counterintuitive as the graphs exhibit very low symmetry. 
On the other hand, this family includes smaller graphs which might contribute to the ease of solving.

\begin{table}[h]
    \centering
    \Large
    \renewcommand{\arraystretch}{1.40}
    \resizebox{\textwidth}{!}{%
    \begin{tabular}{lr ccc ccc ccc ccc ccc}
    \toprule
    Family & Inst. & \multicolumn{3}{c}{Boscia DFS} & \multicolumn{3}{c}{Boscia Clique} & \multicolumn{3}{c}{Boscia Star} & \multicolumn{3}{c}{Boscia Fixings} & \multicolumn{3}{c}{MIP} \\
    \cmidrule(lr){3-5} \cmidrule(lr){6-8} \cmidrule(lr){9-11} \cmidrule(lr){12-14} \cmidrule(lr){15-17}
    & & \# & \% & Time (s) & \# & \% & Time (s) & \# & \% & Time (s) & \# & \% & Time (s) & \# & \% & Time (s) \\
    \midrule
    
    latin & 63 & 13 & 21\% & 719.69 & 63 & \textbf{100\%} & 0.03 & 63 & \textbf{100\%} & 0.51 & 17 & 27\% & 981.22 & 22 & 35\% & 772.60 \\
    Lattice & 21 & 4 & 19\% & 789.79 & 21 & \textbf{100\%} & 0.00 & 21 & \textbf{100\%} & 0.04 & 6 & 29\% & 965.80 & 10 & 48\% & 307.76 \\
    paley\_power & 24 & 6 & 25\% & 494.10 & 24 & \textbf{100\%} & 0.23 & 24 & \textbf{100\%} & 2.93 & 9 & 38\% & 879.39 & 9 & 38\% & 526.51 \\
    paley\_prime & 21 & 6 & 29\% & 374.47 & 21 & \textbf{100\%} & 0.13 & 21 & \textbf{100\%} & 0.96 & 7 & 33\% & 702.10 & 12 & 57\% & 307.43 \\
    Triangular & 21 & 3 & 14\% & 1170.75 & 21 & \textbf{100\%} & 0.01 & 21 & \textbf{100\%} & 0.09 & 5 & 24\% & 1343.66 & 12 & 57\% & 360.24 \\
    \midrule 
    CHH\_cc & 24 & 7 & 29\% & 582.42 & 24 & \textbf{100\%} & 0.00 & 24 & \textbf{100\%} & 0.01 & 6 & 25\% & 1366.70 & 24 & \textbf{100\%} & 30.55 \\
    tnn & 42 & 12 & 29\% & 867.90 & 42 & \textbf{100\%} & 0.01 & 42 & \textbf{100\%} & 0.04 & 12 & 29\% & 1559.77 & 34 & 81\% & 72.93 \\
    \midrule
    cfi & 12 & 0 & 0\% & 3600.00 & 12 & \textbf{100\%} & 0.01 & 12 & \textbf{100\%} & 0.02 & 0 & 0\% & 3600.00 & 12 & \textbf{100\%} & 135.69 \\
    exact & 21 & 12 & 57\% & 108.25 & 18 & 86\% & 0.11 & 21 & \textbf{100\%} & 0.18 & 9 & 43\% & 636.68 & 12 & 57\% & 60.77 \\
    iso\_r01N & 21 & 12 & 57\% & 250.89 & 21 & \textbf{100\%} & 0.01 & 18 & 86\% & 0.28 & 10 & 48\% & 660.03 & 15 & 71\% & 41.00 \\
    sts & 21 & 0 & 0\% & 3600.00 & 21 & \textbf{100\%} & 0.16 & 21 & \textbf{100\%} & 4.20 & 0 & 0\% & 3600.00 & 3 & 14\% & 2124.66 \\
    usr & 36 & 6 & 17\% & 1559.75 & 18 & 50\% & 130.22 & 36 & \textbf{100\%} & 2.55 & 10 & 28\% & 1402.56 & 15 & 42\% & 853.38 \\
    \bottomrule
    \end{tabular}%
    }
    \caption{Performance summary on the non-isomorphic instances by graph family. Best performance for each family is highlighted in bold.
    The reported times are the geometric mean of the total solving times shifted by 1 second for all instances. This includes timed out instances.}
    \label{tab:non_isomorphic_results_summary}
\end{table}

In addition to the isomorphic instances, we also created non-isomorphic instances by flipping random edges of the graphs.
The results are summarized in Table~\ref{tab:non_isomorphic_results_summary}.
Our proposed method based outperforms the MIP feasibility approach.
In particular, we observe that the presolving method based on substructure detection determines non-isomorphism at the root node for all instances.
We highlight however that having identical spectra is a necessary condition for isomorphism and is not respected by these graphs.
Spectral analysis would therefore also discard them as non-isomorphic.
Investigating the  performance of optimization-based methods on isospectral but non-isomorphic graphs is a promising direction for future work.

Finally, we performed preliminary experiments for the graph matching problem, in which one seeks permutation that minimizes the distance between two graphs.
Most of our preprocessing and early-stopping steps are not directly applicable in this setting and are deactivated.
Thus, the integer linear formulation minimizing the $\ell_1$-norm of the mismatch as the objective exhibits a better performance than Boscia.
This warrants further investigation, in particular with regards to preprocessing techniques.

\section{Conclusion}

In this paper, we develop and evaluate mixed-integer convex formulations and algorithms for the graph isomorphism problem.
Despite the problem being well-studied on theoretical and algorithmic aspects, we showed that a mixed-integer convex approach leveraging 
first-order methods and graph-specific presolving techniques could achieve great performance.

From the experiments, we can conclude that optimization-based methods benefit from a high degree of symmetry as it 
increases the number of solutions.  
To effectively solve GI for asymmetric graphs, presolving techniques that detect substructures and infer variable
fixings from them are crucial. We hope that this insight will provide guidance for the design of specialized GI detection solvers.

The fundamental operations performed in Boscia, namely gradient computation and calls to the Hungarian algorithm, could both be ported to GPU for increased performance, unlike solving methods based on constraint programming or linear optimization, offering a promising avenue for future research.

Considering the encouraging performance of our approach, we plan to extend our work to the graph matching problem,
with a particular focus on preprocessing techniques.

\clearpage
\begin{credits}
\subsubsection{\ackname} Research reported in this paper was partially supported through the Research Campus
Modal funded by the German Federal Ministry of Education
and Research (fund numbers 05M14ZAM,05M20ZBM), the Deutsche Forschungsgemeinschaft (DFG) through the DFG Cluster of
Excellence MATH+ (EXC-2046/1, project ID 390685689, project AA3-15), and the ANR through MIAI Cluster (reference ANR-23-IACL-0006).
The authors also thank Pascal Schweitzer for insightful discussions on GI at SEA 2024 motivating parts of this work.

\subsubsection{\discintname}
The authors have no competing interests to declare that are
relevant to the content of this article.
\end{credits}
\bibliographystyle{splncs04}
\bibliography{references}

@article{FW,
  author = {Frank, M. and Wolfe, P.},
  title = {An algorithm for quadratic programming},
  journal = {Naval Research Logistics Quarterly},
  volume = {3},
  pages = {95--110},
  year = {1956},
  url = {https://onlinelibrary.wiley.com/doi/abs/10.1002/nav.3800030109}
}

@article{CG,
  author = {Levitin, E. and Polyak, B.},
  title = {Constrained minimization methods},
  journal = {USSR Computational Mathematics And Mathematical Physics},
  volume = {6},
  pages = {1--50},
  year = {1966},
  url = {https://www.sciencedirect.com/science/article/pii/0041555366901145}
}

@article{DICG,
  title={Linear-memory and decomposition-invariant linearly convergent conditional gradient algorithm for structured polytopes},
  author={Garber, Dan and Meshi, Ofer},
  journal={Advances in neural information processing systems},
  volume={29},
  year={2016}
}

@article{DCA,
  author = {Pokutta, S.},
  title = {Scalable {DC} Optimization via Adaptive {Frank-Wolfe} Algorithms},
  journal = {arXiv preprint arXiv:2507.17545},
  year = {2025},
  url = {https://arxiv.org/abs/2507.17545}
}

@article{hendrych2023convex,
      title={Convex mixed-integer optimization with {Frank-Wolfe} methods},
      author={Deborah Hendrych and Hannah Troppens and Mathieu Besançon and Sebastian Pokutta},
      journal={Mathematical Programming Computation},
      year={2025},
      volume={17},
      pages={731–757}
}

@article{besanccon2025improved,
  author       = {Mathieu Besan{\c{c}}on and
                  S{\'{e}}bastien Designolle and
                  Jannis Halbey and
                  Deborah Hendrych and
                  Dominik Kuzinowicz and
                  Sebastian Pokutta and
                  Hannah Troppens and
                  Daniel Viladrich Herrmannsdoerfer and
                  Elias Samuel Wirth},
  title        = {Improved Algorithms and Novel Applications of the {{FrankWolfe.jl}} Library},
  journal      = {ACM Transactions on Mathematical Software},
  volume       = {51},
  number       = {4},
  pages        = {29:1--29:33},
  year         = {2025},
  url          = {https://doi.org/10.1145/3765626},
  doi          = {10.1145/3765626},
  bibsource    = {dblp computer science bibliography, https://dblp.org}
}

@inproceedings{hendrych2025secant,
  author       = {Deborah Hendrych and
                  Sebastian Pokutta and
                  Mathieu Besan{\c{c}}on and
                  David Mart{\'{\i}}nez{-}Rubio},
  editor       = {Aarti Singh and
                  Maryam Fazel and
                  Daniel Hsu and
                  Simon Lacoste{-}Julien and
                  Felix Berkenkamp and
                  Tegan Maharaj and
                  Kiri Wagstaff and
                  Jerry Zhu},
  title        = {Secant Line Search for {Frank-Wolfe} Algorithms},
  booktitle    = {Forty-second International Conference on Machine Learning, {ICML}
                  2025, Vancouver, BC, Canada, July 13-19, 2025},
  series       = {Proceedings of Machine Learning Research},
  publisher    = {{PMLR} / OpenReview.net},
  year         = {2025},
  url          = {https://proceedings.mlr.press/v267/hendrych25a.html},
  bibsource    = {dblp computer science bibliography, https://dblp.org}
}

@inproceedings{2023_HendrychBesanconPokutta_Optimalexperimentdesign,
  year = {2024},
  booktitle = {Proceedings of the Symposium on Experimental Algorithms},
  doi = {10.4230/LIPIcs.SEA.2024.16},
  archiveprefix = {arXiv},
  eprint = {2312.11200},
  primaryclass = {math.OC},
  author = {Hendrych, Deborah and Besançon, Mathieu and Pokutta, Sebastian},
  title = {Solving the Optimal Experiment Design Problem with {Mixed-Integer} Convex Methods},
  code = {https://github.com/ZIB-IOL/OptimalDesignWithBoscia}
}

@article{besanccon2022frankwolfe,
author = {Besan\c{c}on, Mathieu and Carderera, Alejandro and Pokutta, Sebastian},
title = {{FrankWolfe.jl}: A High-Performance and Flexible Toolbox for {Frank-Wolfe} Algorithms and Conditional Gradients},
journal = {INFORMS Journal on Computing},
volume = {34},
number = {5},
pages = {2611-2620},
year = {2022},
doi = {10.1287/ijoc.2022.1191},
URL = { https://doi.org/10.1287/ijoc.2022.1191},
eprint = {https://doi.org/10.1287/ijoc.2022.1191}
}

@article{braun2017lazifying,
  author       = {G{\'{a}}bor Braun and
                  Sebastian Pokutta and
                  Daniel Zink},
  title        = {Lazifying Conditional Gradient Algorithms},
  journal      = {Journal of Machine Learning Research},
  volume       = {20},
  pages        = {71:1--71:42},
  year         = {2019},
  url          = {https://jmlr.org/papers/v20/18-114.html},
  bibsource    = {dblp computer science bibliography, https://dblp.org}
}

@article{tsuji2021sparser,
  title={Sparser kernel herding with pairwise conditional gradients without swap steps},
  author={Tsuji, Kazuma and Tanaka, Ken'ichiro and Pokutta, Sebastian},
  journal={arXiv preprint arXiv:2110.12650},
  year={2021}
}

@article{lacoste2015global,
  author       = {Simon Lacoste{-}Julien and
                  Martin Jaggi},
  title        = {On the Global Linear Convergence of {Frank-Wolfe} Optimization Variants},
  journal      = {Advances in neural information processing systems},
  volume       = {28},
  year         = {2015},
  url          = {https://proceedings.neurips.cc/paper_files/paper/2015/file/c058f544c737782deacefa532d9add4c-Paper.pdf},
}

@book{braun2022conditional,
	arxiv = {https://arxiv.org/abs/2211.14103},
	author = {Braun, G{\'a}bor and Carderera, Alejandro and Combettes, Cyrille W. and Hassani, Hamed and Karbasi, Amin and Mokthari, Aryan and Pokutta, Sebastian},
	featured = 1,
	pdfurl = {https://doi.org/10.1137/1.9781611978568},
  	publisher={MOS-SIAM Series on Optimization},
	month = {1},
	ptype = {book},
	isbn = {978-1-61197-855-1},
	title = {Conditional Gradient Methods},
	tags = {opt, ml, survey},
	year={2025}
}

@article{pena2019polytope,
  title={Polytope conditioning and linear convergence of the {Frank-Wolfe} algorithm},
  author={Pena, Javier and Rodriguez, Daniel},
  journal={Mathematics of Operations Research},
  volume={44},
  number={1},
  pages={1--18},
  year={2019},
  publisher={INFORMS}
}

@article{klus2025continuous,
  title={Continuous optimization methods for the graph isomorphism problem},
  author={Klus, Stefan and Gel{\ss}, Patrick},
  journal={Information and Inference: A Journal of the IMA},
  volume={14},
  number={2},
  pages={iaaf011},
  year={2025},
  publisher={Oxford University Press}
}

@article{bestuzheva2023enabling,
  author       = {Ksenia Bestuzheva and
                  Mathieu Besan{\c{c}}on and
                  Weikun Chen and
                  Antonia Chmiela and
                  Tim Donkiewicz and
                  Jasper van Doornmalen and
                  Leon Eifler and
                  Oliver Gaul and
                  Gerald Gamrath and
                  Ambros M. Gleixner and
                  Leona Gottwald and
                  Christoph Graczyk and
                  Katrin Halbig and
                  Alexander Hoen and
                  Christopher Hojny and
                  Rolf van der Hulst and
                  Thorsten Koch and
                  Marco E. L{\"{u}}bbecke and
                  Stephen J. Maher and
                  Frederic Matter and
                  Erik M{\"{u}}hmer and
                  Benjamin M{\"{u}}ller and
                  Marc E. Pfetsch and
                  Daniel Rehfeldt and
                  Steffan Schlein and
                  Franziska Schl{\"{o}}sser and
                  Felipe Serrano and
                  Yuji Shinano and
                  Boro Sofranac and
                  Mark Turner and
                  Stefan Vigerske and
                  Fabian Wegscheider and
                  Philipp Wellner and
                  Dieter Weninger and
                  Jakob Witzig},
  title        = {Enabling Research through the {SCIP} Optimization Suite 8.0},
  journal      = {ACM Transactions on Mathematical Software},
  volume       = {49},
  number       = {2},
  pages        = {22:1--22:21},
  year         = {2023},
  url          = {https://doi.org/10.1145/3585516},
  doi          = {10.1145/3585516},
  bibsource    = {dblp computer science bibliography, https://dblp.org}
}

@techreport{BolusaniEtal2024OO,
  author = {Suresh Bolusani and Mathieu Besan{\c{c}}on and Ksenia Bestuzheva and Antonia Chmiela and Jo{\~{a}}o Dion{\'{i}}sio and Tim Donkiewicz and Jasper van Doornmalen and Leon Eifler and Mohammed Ghannam and Ambros Gleixner and Christoph Graczyk and Katrin Halbig and Ivo Hedtke and Alexander Hoen and Christopher Hojny and Rolf van der Hulst and Dominik Kamp and Thorsten Koch and Kevin Kofler and Jurgen Lentz and Julian Manns and Gioni Mexi and Erik~M\"{u}hmer and Marc E. Pfetsch and Franziska Schl{\"o}sser and Felipe Serrano and Yuji Shinano and Mark Turner and Stefan Vigerske and Dieter Weninger and Lixing Xu},
  title = {{The SCIP Optimization Suite 9.0}},
  type = {Technical Report},
  institution = {Optimization Online},
  month = {February},
  year = {2024},
  url = {https://optimization-online.org/2024/02/the-scip-optimization-suite-9-0/}
}

@article{kuhn1955hungarian,
  title={The {Hungarian} method for the assignment problem},
  author={Kuhn, Harold W},
  journal={Naval research logistics quarterly},
  volume={2},
  number={1-2},
  pages={83--97},
  year={1955},
  publisher={Wiley Online Library}
}

@article{munkres1957algorithms,
  title={Algorithms for the assignment and transportation problems},
  author={Munkres, James},
  journal={Journal of the society for industrial and applied mathematics},
  volume={5},
  number={1},
  pages={32--38},
  year={1957},
  publisher={SIAM}
}

@article{Birkhoff_theorem,
  author = {Read, R. and Corneil, D.},
  title = {The graph isomorphism disease},
  journal = {Journal Of Graph Theory},
  volume = {1},
  pages = {339--363},
  year = {1977},
  url = {https://onlinelibrary.wiley.com/doi/abs/10.1002/jgt.3190010410}
}

@article{GI_relaxation,
  author = {Grohe, M. and Schweitzer, P.},
  title = {The graph isomorphism problem},
  journal = {Communications of the ACM},
  volume = {63},
  number = {10},
  pages = {128--134},
  year = {2020},
  month = {10},
  doi = {10.1145/3372123},
  url = {https://doi.org/10.1145/3372123}
}

@article{read1977graph,
  title={The graph isomorphism disease},
  author={Read, Ronald C and Corneil, Derek G},
  journal={Journal of graph theory},
  volume={1},
  number={4},
  pages={339--363},
  year={1977},
  publisher={Wiley Online Library}
}

@article{zemlyachenko1985graph,
  title={Graph isomorphism problem},
  author={Zemlyachenko, Viktor N and Korneenko, Nickolay M and Tyshkevich, Regina I},
  journal={Journal of Soviet Mathematics},
  volume={29},
  number={4},
  pages={1426--1481},
  year={1985},
  publisher={Springer}
}

@inproceedings{babai2018group,
  title={Group, graphs, algorithms: the graph isomorphism problem},
  author={Babai, L{\'a}szl{\'o}},
  booktitle={Proceedings of the International Congress of Mathematicians: Rio de Janeiro 2018},
  pages={3319--3336},
  year={2018},
  organization={World Scientific}
}

@inproceedings{hopcroft1972isomorphism,
  title={Isomorphism of planar graphs},
  author={Hopcroft, John E and Tarjan, Robert Endre},
  booktitle={Complexity of Computer Computations: Proceedings of a symposium on the Complexity of Computer Computations, held March 20--22, 1972, at the IBM Thomas J. Watson Research Center, Yorktown Heights, New York, and sponsored by the Office of Naval Research, Mathematics Program, IBM World Trade Corporation, and the IBM Research Mathematical Sciences Department},
  pages={131--152},
  year={1972},
  organization={Springer}
}

@article{luks1982isomorphism,
  author       = {Eugene M. Luks},
  title        = {Isomorphism of Graphs of Bounded Valence can be Tested in Polynomial
                  Time},
  journal      = {Journal of Circuits, Systems and Computers},
  volume       = {25},
  number       = {1},
  pages        = {42--65},
  year         = {1982},
  url          = {https://doi.org/10.1016/0022-0000(82)90009-5},
  doi          = {10.1016/0022-0000(82)90009-5},
  bibsource    = {dblp computer science bibliography, https://dblp.org}
}

@inproceedings{babai1982isomorphism,
  title={Isomorphism of graphs with bounded eigenvalue multiplicity},
  author={Babai, L{\'a}szl{\'o} and Grigoryev, D Yu and Mount, David M},
  booktitle={Proceedings of the fourteenth annual ACM symposium on Theory of computing},
  pages={310--324},
  year={1982}
}

@article{KS18,
    author  = {S. Klus and T. Sahai},
    title   = {A Spectral Assignment Approach for the Graph Isomorphism Problem},
    journal = {Information and Inference: A Journal of the IMA},
    volume  = {7},
    issue   = {4},
    pages   = {689--706},
    year    = {2018},
    doi     = {10.1093/imaiai/iay001},
}

@article{ABK15,
    author  = {Y. Aflalo and A. Bronstein and R. Kimmel},
    title   = {On convex relaxation of graph isomorphism},
    journal = {Proceedings of the National Academy of Sciences},
    volume  = {112},
    number  = {10},
    pages   = {2942--2947},
    year    = {2015},
}

@article{mckay2014practical,
  author       = {Brendan D. McKay and
                  Adolfo Piperno},
  title        = {Practical graph isomorphism, {II}},
  journal      = {Journal of Symbolic Computation},
  volume       = {60},
  pages        = {94--112},
  year         = {2014},
  url          = {https://doi.org/10.1016/j.jsc.2013.09.003},
  doi          = {10.1016/J.JSC.2013.09.003},
  bibsource    = {dblp computer science bibliography, https://dblp.org}
}

@article{mckay1981practical,
  title={Practical graph isomorphism},
  author={McKay, Brendan D and others},
  year={1981},
  publisher={Department of Computer Science, Vanderbilt University Tennessee, USA}
}

@article{gleixner2017three,
  title={Three enhancements for optimization-based bound tightening},
  author={Gleixner, Ambros M and Berthold, Timo and M{\"u}ller, Benjamin and Weltge, Stefan},
  journal={Journal of Global Optimization},
  volume={67},
  number={4},
  pages={731--757},
  year={2017},
  publisher={Springer}
}

@misc{nauty-traces-benchmarks,
  author       = {Brendan McKay and Adolfo Piperno},
  title        = {nauty and Traces Benchmark Graph Libraries},
  howpublished = {\url{https://pallini.di.uniroma1.it/Graphs.html}},
  note         = {Collections of benchmark graphs (DIMACS and dreadnaut formats); current version 2\_9\_1 (Sep 7, 2025)},
  year         = {2025}
}

@misc{pace-2023-twinwidth,
  title        = {{PACE} Challenge 2023: Twinwidth},
  howpublished = {\url{https://pacechallenge.org/2023/}},
  note         = {{PACE} 2023 challenge on computing low-width contraction sequences; exact and heuristic tracks},
  year         = {2023},
  organization = {Parameterized Algorithms and Computational Experiments (PACE)}
}

@article{maskan2025revisiting,
  title={Revisiting {Frank-Wolfe} for Structured Nonconvex Optimization},
  author={Maskan, Hoomaan and Hou, Yikun and Sra, Suvrit and Yurtsever, Alp},
  journal={arXiv preprint arXiv:2503.08921},
  year={2025}
}

@article{eppstein2013listing,
  author       = {David Eppstein and
                  Maarten L{\"{o}}ffler and
                  Darren Strash},
  title        = {Listing All Maximal Cliques in Large Sparse Real-World Graphs},
  journal      = {ACM Journal of Experimental Algorithmics},
  volume       = {18},
  year         = {2013},
  url          = {https://doi.org/10.1145/2543629},
  doi          = {10.1145/2543629},
  bibsource    = {dblp computer science bibliography, https://dblp.org}
}

@article{fredman1987fibonacci,
  title={Fibonacci heaps and their uses in improved network optimization algorithms},
  author={Fredman, Michael L and Tarjan, Robert Endre},
  journal={Journal of the ACM (JACM)},
  volume={34},
  number={3},
  pages={596--615},
  year={1987},
  publisher={ACM New York, NY, USA}
}

@article{aurora2018qap,
  title={The {QAP}-polytope and the graph isomorphism problem},
  author={Aurora, Pawan and Mehta, Shashank K},
  journal={Journal of Combinatorial Optimization},
  volume={36},
  number={3},
  pages={965--1006},
  year={2018},
  publisher={Springer}
}

@article{valls2021birkhoff,
  title={Birkhoff’s decomposition revisited: Sparse scheduling for high-speed circuit switches},
  author={Valls, V{\'\i}ctor and Iosifidis, George and Tassiulas, Leandros},
  journal={IEEE/ACM Transactions on Networking},
  volume={29},
  number={6},
  pages={2399--2412},
  year={2021},
  publisher={IEEE}
}

% ---- Appendix ----
\appendix
\section{Cospectral Non-Isomorphic Graphs results}
\label{app:additional-results}

\begin{table}[h]
    \centering
    \Large
    \renewcommand{\arraystretch}{1.40}
    \resizebox{\textwidth}{!}{%
    \begin{tabular}{l ccc ccc ccc ccc ccc}
    \toprule
    Family & \multicolumn{3}{c}{Boscia DFS} & \multicolumn{3}{c}{Boscia Clique} & \multicolumn{3}{c}{Boscia Star} & \multicolumn{3}{c}{Boscia Fixings} & \multicolumn{3}{c}{MIP} \\
    \cmidrule(lr){2-4} \cmidrule(lr){5-7} \cmidrule(lr){8-10} \cmidrule(lr){11-13} \cmidrule(lr){14-16}
    & \# & \% & Time (s) & \# & \% & Time (s) & \# & \% & Time (s) & \# & \% & Time (s) & \# & \% & Time (s) \\
    \midrule
    
    SRG29 & 0 & 0\% & 3600.000 & 45 & \textbf{100\%} & \textbf{0.002} & 45 & \textbf{100\%} & 0.017 & 45 & \textbf{100\%} & 417.002 & 45 & \textbf{100\%} & 112.987 \\
    SRG35 & 0 & 0\% & 3600.000 & 45 & \textbf{100\%} & \textbf{0.003} & 45 & \textbf{100\%} & 0.024 & 45 & \textbf{100\%} & 885.617 & 45 & \textbf{100\%} & 419.530 \\
    SRG36 & 0 & 0\% & 3600.000 & 45 & \textbf{100\%} & \textbf{0.004} & 45 & \textbf{100\%} & 0.019 & 45 & \textbf{100\%} & 973.940 & 45 & \textbf{100\%} & 1200.798 \\
    SRG37 & 0 & 0\% & 3600.000 & 45 & \textbf{100\%} & \textbf{0.004} & 45 & \textbf{100\%} & 0.037 & 42 & 93.3\% & 1590.834 & 45 & \textbf{100\%} & 1213.421 \\
    SRG40 & 0 & 0\% & 3600.000 & 45 & \textbf{100\%} & \textbf{0.002} & 45 & \textbf{100\%} & 0.024 & 45 & \textbf{100\%} & 1984.101 & 45 & \textbf{100\%} & 45.586 \\
    \midrule
    HIG11 & 0 & 0\% & 3600.000 & 45 & \textbf{100\%} & \textbf{0.003} & 45 & \textbf{100\%} & 0.019 & 45 & \textbf{100\%} & 977.676 & 45 & \textbf{100\%} & 822.373 \\
    HIG12 & 0 & 0\% & 3600.000 & 45 & \textbf{100\%} & \textbf{0.004} & 45 & \textbf{100\%} & 0.019 & 45 & \textbf{100\%} & 945.132 & 45 & \textbf{100\%} & 839.970 \\
    HIG13 & 0 & 0\% & 3600.000 & 45 & \textbf{100\%} & \textbf{0.003} & 45 & \textbf{100\%} & 0.019 & 45 & \textbf{100\%} & 961.563 & 45 & \textbf{100\%} & 743.829 \\
    HIG14 & 0 & 0\% & 3600.000 & 45 & \textbf{100\%} & \textbf{0.002} & 45 & \textbf{100\%} & 0.021 & 45 & \textbf{100\%} & 1627.581 & 45 & \textbf{100\%} & 736.427 \\
    HIG15 & 0 & 0\% & 3600.000 & 45 & \textbf{100\%} & \textbf{0.004} & 45 & \textbf{100\%} & 0.017 & 45 & \textbf{100\%} & 973.605 & 45 & \textbf{100\%} & 686.387 \\
    \bottomrule
    \end{tabular}
    }
    \caption{Performance summary on cospectral non-isomorphic graph instances by graph family.
    SRG refers to Strongly Regular Graphs and HIG to Highly Irregular Graphs, with the numerical suffix indicating the number of nodes. Each family contains of 45 instances.}
    \label{tab:cospectral_results}
    \end{table}

\end{document}